\documentclass[10pt]{amsart}
\usepackage[a4paper,top=4cm,bottom=4cm,left=2.5cm,right=2.5cm]{geometry}
\usepackage{amsthm}
\usepackage[colorlinks=true,urlcolor=blue,citecolor=red,linkcolor=blue,linktocpage,pdfpagelabels,bookmarksnumbered,bookmarksopen]{hyperref}
\usepackage[english]{babel}
\usepackage{graphicx}
\usepackage{fancyhdr}
\usepackage{setspace}
\usepackage{xcolor}
\usepackage{amssymb,amsmath,mathtools}
\usepackage{physics}
\usepackage{esint}
\usepackage{latexsym}
\usepackage[mathscr]{eucal}
\usepackage{mathrsfs}
\usepackage{enumitem}
\usepackage{microtype}

\DeclareMathOperator\dist{dist}
\allowdisplaybreaks
\addtocontents{toc}{\protect\setcounter{tocdepth}{1}}
\def\mean_#1{\mathchoice%
          {\mathop{\kern 0.2em\vrule width 0.6em height 0.69678ex depth -0.58065ex
                  \kern -0.8em \intop}\nolimits_{\kern -0.4em#1}}%
          {\mathop{\kern 0.1em\vrule width 0.5em height 0.69678ex depth -0.60387ex
                  \kern -0.6em \intop}\nolimits_{#1}}%
          {\mathop{\kern 0.1em\vrule width 0.5em height 0.69678ex
              depth -0.60387ex
                  \kern -0.6em \intop}\nolimits_{#1}}%
          {\mathop{\kern 0.1em\vrule width 0.5em height 0.69678ex depth -0.60387ex
                  \kern -0.6em \intop}\nolimits_{#1}}}

\theoremstyle{plain}
\newtheorem{Theorem}{Theorem}[section]
\newtheorem{Proposition}[Theorem]{Proposition}
\newtheorem{Lemma}[Theorem]{Lemma}
\newtheorem{Corollary}[Theorem]{Corollary}
\theoremstyle{definition}
\newtheorem{Definition}[Theorem]{Definition}
\newtheorem{Remark}[Theorem]{Remark}

\DeclareMathSymbol{\Finv}{\mathord}{AMSb}{"60}
\numberwithin{equation}{section}
\newcommand{\N}{\mathbb{N}}

\newcommand{\R}{\mathbb{R}}

\newcommand{\Div}{\operatorname{div}}

\newcommand{\loc}{\mathrm{loc}}

\newcommand{\capM}{\operatorname{cap}_{M}}
\newcommand{\capD}{\operatorname{cap}_{D}}
\newcommand{\CAPD}{\operatorname{CAP}_{D}}

\newcommand{\WM}{W_M}
\newcommand{\VD}{V_D(\Omega)}
\newcommand{\WD}{W_D}
\newcommand{\WtD}{\widetilde W_D}
\newcommand{\VM}{V_M}
\newcommand{\EnergyM}{\mathcal{E}_M}
\newcommand{\Dp}{d_p}
\newcommand{\calH}{\mathcal H}
\newcommand{\supp}{\operatorname{spt}}

\newcommand{\homnorm}[1]{\lVert #1\rVert_{W_M}}

\title[Duality Conditions for Morrey Measures via Parabolic Capacity]{Duality Conditions for Morrey Measures \\via Parabolic Capacity}
\author{Lorenzo Braglia}
\address[L. Braglia]{Dipartimento di Matematica e Informatica, Università degli Studi di Ferrara, Via Machiavelli 35, 44121 Ferrara, Italy}
\email{lorenzo.braglia@unife.it}
\date{\today}

\begin{document}

\begin{abstract}
We establish Morrey-type conditions ensuring that a finite signed
Radon measure belongs to the dual of the energy space
of solutions to nonlinear parabolic equations of $p$-Laplacian type. More precisely, for the
parabolic cylinders
\[
Q_{r,r^p}(z):=B_r(x)\times(t-r^p,t+r^p),
\]
we prove that
\[
|\mu|\bigl(Q_{r,r^p}(z)\cap(\Omega\times (0,T))\bigr)
\leq Mr^{n+p-\vartheta},
\qquad
\vartheta<p,
\]
implies duality, and this threshold is sharp in general. More generally, for cylinders with
time length proportional to $r^q$, $q>1$, the sufficient threshold is
$\vartheta<\min\{p,q\}$. Finally, we discuss the resulting variational, energy, and
renormalized solution theories for nonlinear parabolic equations with
measure data, relating our conclusions to the existing literature.
\end{abstract}

\maketitle
\section{Introduction}

The model problem motivating this work is the Cauchy--Dirichlet problem associated with the non-homogeneous equation 
\[
\partial_tu-\Div a(x,t,\nabla u)=\mu\qquad\text{in $\Omega_T=\Omega\times(0,T)$}
\]
where $\Omega$ is a bounded Lipschitz domain of $\R^n$, $\mu$ is a finite signed Radon measure on $\Omega_T$ and the vector field $a$ has $p$-growth, with prototype
\begin{equation}\label{p.laplacian}
a(x,t,\nabla u)= \bigl(s^2+|\nabla u|^2\bigr)^{\frac{p-2}{2}}\nabla u
\end{equation}
for $1<p<\infty$, see Section \ref{sec:pde-consequences} for the precise assumptions. 

\medskip

Under suitable assumptions on the right-hand side $\mu$, the problem
can be formulated in the natural energy spaces associated with the
equation, and solutions can be constructed by direct monotonicity
methods. In general, however, a finite Radon measure $\mu$ need not satisfy a duality condition such as
\begin{equation}\label{lionsrep}
\mu\in L^{p'}(0,T;W^{-1,p'}(\Omega)),\qquad W^{-1,p'}(\Omega):= (W_0^{1,p}(\Omega))';
\end{equation}
our aim is to identify density conditions on $\mu$ ensuring that  it defines a continuous linear functional on the parabolic energy space $\WM(\Omega_T)$ introduced below. Under condition \eqref{lionsrep}, \cite[Chapter~II, Theorem~1.2 bis]{JLL} yields a unique energy solution in $L^p(0,T;W^{1,p}_0(\Omega))$ for every initial datum $u_0\in L^2(\Omega)$. To apply this result throughout the full range $1<p<\infty$, we take 
 \[
V:=W^{1,p}_0(\Omega)\cap L^2(\Omega),\qquad H:=L^2(\Omega)
\]
and we obtain $u\in L^p(0,T;V)\cap C([0,T];L^2(\Omega))$. Moreover, the growth assumptions in \eqref{p.laplacian} (see the forthcoming \eqref{assumptionsa}), together with \eqref{lionsrep}, imply through the equation that $\partial_tu\in L^{p'}(0,T;W^{-1,p'}(\Omega))$, rather than merely in the {\em a priori} larger space $L^{p'}(0,T;V')$ dictated by the choice of $V$. We retain this sharper information in the definition of $W_M(\Omega_T)$ below. Since duality reverses inclusions, pinning $\partial_t u$ to the smaller space $W^{-1,p'}(\Omega)$ keeps $W_M(\Omega_T)$ correspondingly smaller, and therefore enlarges its dual: this is precisely the space needed to accommodate the wider class of measures $\mu$ we aim to treat. This is a genuine gain only for $1<p<2n/(n+2)$, where the space $W^{-1,p'}(\Omega)$ is strictly smaller than $V'$; for $p\ge 2n/(n+2)$, the two spaces for the time derivative coincide, see the last paragraph in Section \ref{sec:duality}. We therefore set
\begin{equation*}
\WM(\Omega_T)
:=
\left\{
 u\in L^p(0,T;W^{1,p}_0(\Omega))\cap L^\infty(0,T;L^2(\Omega)):
 \partial_tu\in L^{p'}(0,T;W^{-1,p'}(\Omega))
\right\}
\end{equation*}
and we naturally endow the space with the homogeneous Banach norm
\begin{equation}\label{eq:intro-WM-norm}
\homnorm{u}
:=
\|\nabla u\|_{L^p(\Omega_T)}
+
\|\partial_tu\|_{L^{p'}(0,T;W^{-1,p'}(\Omega))}
+
\|u\|_{L^\infty(0,T;L^2(\Omega))}.
\end{equation}
Condition \eqref{lionsrep} entails $\mu\in W'_M(\Omega_T)$ a fortiori, but the latter also accommodates measures failing \eqref{lionsrep}, that is the class of interest here; for such data, the theory of \cite{DPP} provides existence and uniqueness of variational solutions for $u_0\in L^2(\Omega)$. For finite measures that do not charge sets of zero parabolic capacity, the same work establishes existence and uniqueness of renormalized solutions for $u_0\in L^1(\Omega)$; we discuss these solution frameworks and their relation to our duality criterion in Section \ref{sec:pde-consequences}. 

\medskip

Consequently, a useful geometric criterion for a measure should ideally answer two questions at once:
\begin{enumerate}
\item does the measure act continuously on the parabolic energy space $W_M(\Omega_T)$?
\item does it avoid the polar (i.e., the negligible) sets of the associated parabolic capacity?
\end{enumerate}
The Morrey condition of Definition \ref{ass:intrinsic-morrey} provides affirmative answers to both questions, and therefore will be the object of this study. 

\subsection{Elliptic duality and nonlinear potentials}

The stationary problem already shows why explicit duality conditions are
valuable. For an elliptic operator with $p$-growth, the natural energy space is
$W^{1,p}_0(\Omega)$ and its data space is
$W^{-1,p'}(\Omega)$. Sobolev embedding gives elementary sufficient
conditions: if $1<p<n$, then
\[
L^{(p^*)'}(\Omega)\hookrightarrow W^{-1,p'}(\Omega),
\qquad
p^*:=\frac{np}{n-p}.
\]
If $p=n$, every $L^m$ datum with $m>1$ acts continuously on $W_0^{1,p}(\Omega)$, while for $p>n$
the embedding $W^{1,p}_0(\Omega)\hookrightarrow C^{0,1-\frac{n}{p}}(\overline\Omega)$ permits
every finite Radon measure to act on the energy space.

For positive measures there are much finer criteria. In the range
$1<p\le n$, the nonlinear potential characterization can be written as
\begin{equation}\label{eq:intro-Wolff}
\bar \mu\in W^{-1,p'}(\R^n)
\quad\Longleftrightarrow\quad
\int_\Omega \mathbf W^{\bar \mu}_{1,p}(x,R)\,d\mu(x)=\int_\Omega\, \int_0^R
\left(\frac{\bar \mu(B_\varrho(x))}{\varrho^{n-p}}\right)^{1/(p-1)}\,\frac{d\varrho}{\varrho}\,d\mu(x)<\infty,
\end{equation}
where $R>0$ is fixed and $\bar \mu$ denotes the zero extension of $\mu$ to $\R^n\smallsetminus \Omega$, see \cite{HW} or \cite[Theorem 4.7.5]{Z}. More general trace inequalities and their
nonlinear-potential characterizations were developed by Cascante, Ortega and Verbitsky in \cite{COV}, and by Verbitsky in \cite[Theorem 1.13]{V}. The elliptic Morrey bound
\[
\mu(B_r(x))\le C r^{n-\vartheta},
\qquad
\vartheta<p,
\]
is precisely the power condition suggested by \eqref{eq:intro-Wolff}. For $1<p<n$, Adams' trace inequality actually yields the stronger embedding $W^{1,p}_0(\Omega)\hookrightarrow L^q(\Omega,d\mu)$, with $q=p(n-\vartheta)/(n-p)>p$; see \cite[Theorem~4.7.2]{Z}. Membership in $W^{-1,p'}(\Omega)$ requires only the corresponding $L^1(d\mu)$ trace estimate, so the above Morrey condition is sufficient but not necessary for duality. For instance, let $\Omega=B_1(0)$ and $1<p<n$: the measure $d\mu=|x|^{-p}\,dx$ belongs to $W^{-1,p'}(B_1)$, since $|x|^{-p}\in L^{(p^*)'}(B_1)$. However,
\[
\mu(B_r(0))=c(n,p)r^{n-p},\qquad 0<r<1,
\]
which rules out the above Morrey bound for every
$\vartheta<p$.

\medskip

The parabolic problem is more delicate. Although a potential-theoretic characterization of measure data in $\WM'(\Omega_T)$ comparable to \eqref{eq:intro-Wolff} still appears to be out of reach to our knowledge, some recent works have begun to address related aspects; see for example \cite{AdSF}. The authors of \cite{DPP} introduced a functional capacity adapted to nonlinear evolution problems and used it to characterize soft measures and renormalized solutions. Kinnunen, Korte, Kuusi and Parviainen defined a nonlinear capacity through measure-data problems and capacitary potentials; see \cite{KKKP} and also \cite{AKP}. More recently, Moring and Scheven in \cite{MS} constructed a variational capacity and established its relation with nonlinear capacity and with parabolic Hausdorff measure. The present work contributes to this line of investigation.

\subsection{Morrey scales and the main results}

For $q>1$, $z=(x,t)\in\R^{n+1}$, and $r>0$, we use the notation
\begin{equation*}
Q_{r,r^q}(z):=B_r(x)\times(t-r^q,t+r^q);
\end{equation*}
these cylinders are the balls of the metric
\begin{equation}\label{eq:dp}
d_q((x,t),(y,s)):=\max\big\{|x-y|,|t-s|^{1/q}\big\}.
\end{equation}
We call parabolic cylinders those of the form $Q_{r,r^p}(z)$, while caloric cylinders are the ones corresponding to $Q_{r,r^2}(z)$. In what follows, we always specify the time scale we are considering.
\begin{Definition}[Morrey condition]\label{ass:intrinsic-morrey}
Let $1<q<\infty$ and let $\mu$ be a finite signed Radon measure on $\Omega_T$. We say that $\mu$
satisfies the Morrey condition (on $(r,r^q)-$cylinders) with parameter $\vartheta$, if there
exists $M>0$ such that
\begin{equation}\label{eq:intrinsic-morrey}
|\mu|\bigl(Q_{r,r^q}(z)\cap\Omega_T\bigr)
\le M r^{n+q-\vartheta},
\end{equation}
for every $z\in\Omega_T$ and every $r>0$.
\end{Definition}
We now state the main result of this paper.
\begin{Theorem}\label{thm:intrinsic-duality}
Let $1<p,q<\infty$ and let $\mu$ be a finite signed Radon measure on a bounded open set $\Omega_T$ satisfying \eqref{eq:intrinsic-morrey} on $(r,r^q)$-cylinders with $\vartheta<\min\{p,q\}$, and let $\Omega$ be a Lipschitz domain. Then, denoting by $\widetilde u$ the quasi-continuous representative associated with the parabolic capacity (see Theorem \ref{thm:DPP-representative}),
\begin{equation}\label{maingoal}
\biggl|\int_{\Omega_T}\widetilde u\,d\mu\biggr|
\leq
C\|u\|_{\WM(\Omega_T)},\quad\text{for every } u\in\WM(\Omega_T),
\end{equation}
where $C$ depends only on
$n,p,q,\Omega,T,\vartheta,M$, and $|\mu|(\Omega_T)$. Consequently, the map
\[
L_\mu(u):=\int_{\Omega_T}\widetilde u\,d\mu
\]
is a well-defined continuous linear functional on
$\WM(\Omega_T)$, and in particular $\mu\in\WM'(\Omega_T)$.
\end{Theorem}

When $q\ge p$, the strict threshold $\vartheta<p$ cannot be relaxed in general; see Remark~\ref{rem:sharpness}.

\medskip

\begin{Remark}[On the regularity assumption on the domain]
\label{rem:weaker-boundary-assumption}
The Lipschitz regularity of $\partial\Omega$ is a convenient, but
not essential, assumption for the main duality theorem. In fact, its proof uses
the geometry of the boundary to apply results 
valid under the weaker assumption that $\R^n\setminus\Omega$ is uniformly $p$-thick; see \cite[Definition 2.9]{MS}.
More precisely, in the proof, we use this assumption to invoke the following results:
first, the boundary Poincar\'e inequality used in
Lemma \ref{lem:boundary-poincare}; see, in particular, \cite[Theorem~3.7]{KK}. Second, it is
the geometric hypothesis required in \cite[Theorem 1.2]{MS}.
Third, uniform $p$-thickness implies
the Hardy inequality needed in the proof of Density Lemma \ref{lem:smooth-density-WM}; see \cite{JL}. We retain the stronger and
more familiar Lipschitz assumption in the statements in order to
avoid introducing the definition of uniform
$p$-thickness and the associated notation.
\end{Remark}

\subsection{Morrey conditions in parabolic problems}
The duality criterion above also provides a useful point of comparison with the role of Morrey conditions in parabolic regularity theory. These are natural because they measure how strongly the datum concentrates at small scales, and in this regard their relation with regularity of parabolic equations of $p$-Laplacian type was first studied in \cite{B14a} for degenerate parabolic equations, while the singular cases are considered in \cite{B17} and \cite{Pa}; see also \cite{BD} for a study on non-smooth domains. All these papers draw inspiration from the seminal papers by Mingione \cite{M1,M2}, where it is shown how such conditions improve the gradient integrability for solutions of non-linear equations of $p$-Laplacian type; for the Poisson equation, this phenomenon follows from the Newtonian potential representation, which reduces gradient estimates to bounds for Riesz potentials of order one such as those established by Adams \cite{A}.

\medskip

 Consider a datum satisfying the Morrey condition
\[
|\mu|\bigl(B_r(x)\times(t-r^2,t+r^2)\bigr) \le C r^{n+2-\vartheta}.
\]
For degenerate equations, with $p\ge2$, the result of \cite{B14a} yields
\[
|\nabla u|\in\mathcal M^{m}_{\loc}(\Omega_T), \qquad m=p-1+\frac{1}{\vartheta-1},
\]
for $\vartheta_c<\vartheta\le n+2$, where $\vartheta_c\in(1,2)$ depends on the structural data. In particular, $m>p$ whenever $\vartheta_c<\vartheta<2$. A similar connection with the natural energy exponent holds in the singular case. For $n\ge2$ and $2-\frac{1}{n+1}<p<2$, the result of \cite{B17} gives
\[
|\nabla u|\in\mathcal M^{m}_{\loc}(\Omega_T),\qquad m=\frac{(p-1)\vartheta}{\vartheta-1}, \qquad p\le\vartheta\le n;
\]
as observed in \cite[Remark~1.4]{B17}, this estimate also extends to $\vartheta<p$ sufficiently close to $p$, in which case $m>p$. $\mathcal M^m(\Omega_T)$ denotes in both cases the Marcinkiewicz space (also known as weak-$L^m$), defined by
\[
h\in\mathcal M^m(\Omega_T)
\iff
\sup_{\lambda>0}\lambda^m\big|\{z\in\Omega_T:|h(z)|>\lambda\}\big|=:\|h\|_{\mathcal M^m(\Omega_T)}^m<\infty,
\]
for a measurable map $h:\Omega_T\to\R$; the local variant is defined in the usual way. Since weak-$L^m$ spaces embed into $L^p$ on sets of finite measure whenever $m>p$, these estimates yield $\nabla u\in L^p_{\loc}(\Omega_T)$ when $\vartheta<2$ in the degenerate case and $\vartheta<p$ in the singular range above. Thus, in the parameter ranges covered by these results, the gradient reaches the natural energy exponent under the condition $\vartheta<\min\{p,2\}$ and this agrees with the sufficient condition for membership in $\WM'(\Omega_T)$ established here for Morrey data on caloric cylinders. For potential estimates, implying the aforementioned results once the structure of the equation is sufficiently regular, see also \cite{KM13,KM14}.

Beyond gradient regularity, Morrey-type decay estimates have recently been used in \cite{BESS} to prove removability results for degenerate parabolic equations. The argument controls the Riesz measure associated with an obstacle-problem solution without first estimating the solution's oscillation. This illustrates the flexibility of measure decay estimates when the lack of a fixed scaling for the parabolic $p$-Laplacian makes classical oscillation arguments difficult to apply.

\medskip

\subsection{Examples and geometric consequences}
The failure at the endpoint is illustrated by the case
$p=n\ge2$ and $q\ge p$: indeed, the product measure $\delta_{x_0}\otimes\mathcal L^1\llcorner J$, where $J\Subset(0,T)$ is a compact non-degenerate interval, satisfies the endpoint Morrey estimate with $\vartheta=p$, but
charges a set of zero parabolic capacity. Consequently, it does not
belong to $\WM'(\Omega_T)$; see
Remark~\ref{rem:sharpness}. Nevertheless, the Morrey growth condition
is not necessary for duality and therefore does not provide a
characterization of $\WM'(\Omega_T)$. Indeed, 
the following finite and positive Radon measure
\begin{equation}\label{meas.slice}
\mu:=\mathcal{L}^n\otimes\delta_{t_0},\qquad t_0\in(0,T) 
\end{equation}
defines a bounded functional on $\WM(\Omega_T)$, because
$\WM(\Omega_T)\hookrightarrow C([0,T];L^2(\Omega))$, and then
\[
\left|
\int_\Omega u(x,t_0)\,dx
\right|
\le
\mathcal{L}^n(\Omega)^{1/2}\|u(t_0)\|_{L^2(\Omega)}
\le
C\homnorm{u}.
\]
On the other hand, $\mu$ satisfies \eqref{eq:intrinsic-morrey} at the endpoint $\vartheta=q$, since
\[
(\mathcal{L}^n\otimes\delta_{t_0})
\bigl(Q_{r,r^q}(z)\cap\Omega_T\bigr)
\le Cr^n.
\]
For $z=(x,t_0)$ and $B_r(x)\subset\Omega$, this mass is
exactly $c(n)r^n$: hence no Morrey bound with $\vartheta<q$ can hold.

\medskip

The duality criterion also yields the following geometric consequence; here, $\dim^{\mathcal P}_{\mathcal H,p}$ denotes the Hausdorff dimension with respect to the $p$-parabolic metric $d_p$ defined in \eqref{eq:dp}; see also Definition~\ref{HausdorffDim} for $q=p$.

\begin{Proposition}[Parabolic dimension and dual measures]\label{prop:dimension-dual-measures}
Let $E\Subset\Omega_T$ be compact. 
\begin{itemize}

\item If $\dim^{\mathcal P}_{\mathcal H,p}(E)<n$, then $E$ does not support any nonzero finite positive Radon measure belonging to $\WM'(\Omega_T)$.
\item If $\dim^{\mathcal P}_{\mathcal H,p}(E)>n$, then there exists a nonzero finite positive Radon measure $\lambda_E\in\WM'(\Omega_T)$ with $\operatorname{spt}\lambda_E\subset E$.
\end{itemize}
\end{Proposition}

At the critical dimension $n$, no conclusion can be drawn from the dimension alone. Indeed, restricting the time-slice measure considered in \eqref{meas.slice} to $K\times\{t_0\}$, where $K\Subset\Omega$ is compact with positive Lebesgue measure, gives a nonzero positive dual measure supported on a set of parabolic dimension $n$.

On the other hand, when $p=n\ge2$, the vertical segment
$\{x_0\}\times J$ considered above satisfies
$\dim^{\mathcal P}_{\mathcal H,p}(\{x_0\}\times J)=n$.
Nevertheless, such a segment has zero $\capM$-capacity, as shown in Remark~\ref{rem:sharpness}, and therefore it
cannot support any nonzero positive measure in
$\WM'(\Omega_T)$.

\medskip

A further class of examples is obtained from products of
finite positive Radon measures $\nu$ on $\Omega$ and
$\sigma$ on $(0,T)$. Suppose that, for some
$d\in[0,n]$ and $\alpha\in[0,1]$,
\[
\nu\big(B_r(x)\big)\leq C_\nu r^d,
\qquad
\sigma\big((t-\rho,t+\rho)\big)\leq C_\sigma\rho^\alpha;
\]
then $(\nu\otimes\sigma)(Q_{r,r^p}(z))\leq Cr^{d+p\alpha}$
and consequently
\[
d+p\alpha>n
\quad\Longrightarrow\quad
\nu\otimes\sigma\in\WM'(\Omega_T).
\]
In particular, if $E\Subset\Omega$ is compact and nonempty, and $\dim_{\mathcal H}(E)>n-p$,
then there exists a nonzero measure $\nu_E$, supported on
$E$, such that
\[
\nu_E\otimes\mathcal L^1\llcorner J
\in\WM'(\Omega_T)
\]
for every non-degenerate interval $J\Subset(0,T)$. For $p\le n$, this follows from Frostman's lemma by choosing
a spatial growth exponent $d>n-p$; for $p>n$, one may take a Dirac mass supported on $E$. At the critical
spatial dimension $n-p$, this conclusion may fail. Here, $\dim_{\mathcal H}$ denotes the Euclidean Hausdorff
dimension; see \cite[Definition~2.2]{EG}.

\medskip
For absolutely continuous measures, the Morrey condition is not
expected to yield optimal integrability assumptions. Indeed, the parabolic Sobolev inequality \cite[Proposition 3.1]{DiB} gives
\begin{equation}\label{parabsovemb}
W_M(\Omega_T)\hookrightarrow L^{\gamma}(\Omega_T),\qquad \gamma:=
\max\left\{
2,\frac{p(n+2)}{n}
\right\}.  
\end{equation}
Consequently,
\[
h\in L^{\gamma'}(\Omega_T)
\quad\Longrightarrow\quad
\mu:=h\,dx\,dt\in W_M'(\Omega_T).
\]
By contrast, estimating the mass of every parabolic cylinder directly
by H\"older's inequality leads to the more restrictive condition
\[
h\in L^m(\Omega_T),
\quad
m>\frac{n+p}{p} \quad \Longrightarrow\quad  \mu:=h\,dx\,dt\in W_M'(\Omega_T).
\]

Thus the Morrey criterion is not designed to provide optimal conditions for measures having a Lebesgue density. Its main relevance lies in the treatment of singular measures, for which \eqref{parabsovemb} cannot be paired with a density. 

\subsection{Organization of the paper.}
Section~\ref{sec:spaces-capacities} compares the functional spaces and the relevant notions of
capacity. Section~\ref{sec:preliminary} contains the preliminary estimates, including the
interior and boundary Poincar\'e inequalities and the diffuseness
result. Section~\ref{sec:duality} proves the main duality theorem and constructs
the endpoint counterexample. A subsection then records the
simplification available when $p\geq2n/(n+2)$, while Section~\ref{sec:pde-consequences}
discusses variational, energy, and renormalized solutions. Finally, the
Appendix contains the technical proof of Density Lemma \ref{lem:smooth-density-WM}.

\section{Function spaces and parabolic capacity}\label{sec:spaces-capacities}

In this section we review the spaces and capacities introduced in
\cite{DPP,MS}. Throughout this article, we use the notation presented here. Set
\[
\VM(\Omega_T):=L^p(0,T;W^{1,p}_0(\Omega)),
\qquad
\VM'(\Omega_T):=L^{p'}(0,T;W^{-1,p'}(\Omega)).
\]
We endow these Bochner spaces with the norms
\[
\|\varphi \|_{\VM}:=\biggl(\int_{\Omega_T} |\varphi(x,t)|^p+|\nabla\varphi(x,t)|^p\,dx\,dt\biggr)^{1/p}\qquad \|\varphi\|_{V_M'}:=\biggl(\int_0^T \|\varphi(t)\|^{p'}_{W^{-1,p'}(\Omega)}\,dt \biggr)^{1/{p'}};
\]
For the basic theory of Bochner--Lebesgue spaces, we refer to \cite[Chapter~1]{HNVW}. Moreover, in view of the slice-wise Poincaré inequality, we can work with the equivalent norm $\|\varphi\|_{V_M}=\|\nabla \varphi\|_{L^p(\Omega_T)}$.
\begin{Lemma}[Space-time divergence representation]
\label{lem:divergence-representation}
For every $\mathcal G\in\VM'(\Omega_T)$ there exists
$F\in L^{p'}(\Omega_T;\R^n)$ such that $\mathcal G=\Div F$ in the sense of distributions, i.e., 
\begin{equation}\label{eq:divergence-representation}
\begin{aligned}
\langle \mathcal G,\varphi\rangle_{\VM',\VM}
:=
\int_0^T
\langle \mathcal G(t),\varphi(t)\rangle_
{W^{-1,p'}(\Omega),W^{1,p}_0(\Omega)}
\,dt=
-\int_{\Omega_T}F\cdot\nabla\varphi\,dx\,dt,
\end{aligned}
\end{equation}
for every $\varphi\in\VM(\Omega_T)$. 
Moreover,
\begin{equation}\label{eq:negative-norm-representation}
\|\mathcal G\|_{\VM'}
=
\inf\left\{
\|F\|_{L^{p'}(\Omega_T)}:
F\in L^{p'}(\Omega_T;\R^n),\
\mathcal G=\Div F
\right\}.
\end{equation}
\end{Lemma}

\begin{proof}
Consider the linear map
\[
J:\VM(\Omega_T)\to J(\VM(\Omega_T))\subset L^p(\Omega_T;\R^n),
\qquad
J\varphi:=\nabla\varphi.
\]
Since $J$ is a linear isometry, $J(\VM(\Omega_T))$ is a closed linear subspace of $L^p(\Omega_T;\R^n)$. It is also injective: if $J(\varphi_1)=J(\varphi_2)$, then $\varphi_1=\varphi_2$ almost everywhere, due to slice-wise Poincaré.

Define the linear functional
\begin{equation}\label{eq:Lambda-on-gradient-range}
\Lambda:J(\VM(\Omega_T))\to\R,\qquad \Lambda(J\varphi)
:=
-\langle\mathcal G,\varphi\rangle_{\VM',\VM}.
\end{equation}
This functional is well defined, because if $J\varphi_1=J\varphi_2$, then $J(\varphi_1-\varphi_2)=0$. By the injectivity proved above,
$\varphi_1=\varphi_2$ almost everywhere, and therefore the right-hand
side of \eqref{eq:Lambda-on-gradient-range} is independent of the choice of the preimage. Moreover
\[
\begin{aligned}
|\Lambda(J\varphi)|
=
|\langle\mathcal G,\varphi\rangle_{\VM',\VM}|
\le
\|\mathcal G\|_{\VM'}
\|\varphi\|_{\VM}
=
\|\mathcal G\|_{\VM'}
\|J\varphi\|_{L^p(\Omega_T)}.
\end{aligned}
\]
Thus $\Lambda$ is continuous on $J(\VM(\Omega_T))$, with $\|\Lambda\|_{(J(\VM(\Omega_T)))'}\le\|\mathcal G\|_{\VM'}$.
On the other hand
\[
\begin{aligned}
\|\mathcal G\|_{\VM'}
=
\sup_{\|\varphi\|_{\VM}\le1}
|\langle\mathcal G,\varphi\rangle_{\VM',\VM}|
=\sup_{\substack{H\in J(\VM(\Omega_T))\\\|H\|_{L^p(\Omega_T)}\le1}}
|\Lambda(H)|
=
\|\Lambda\|_{(J(\VM(\Omega_T)))'}.
\end{aligned}
\]
Hence, $\|\Lambda\|_{(J(\VM(\Omega_T)))'}=\|\mathcal G\|_{\VM'}$. By applying the Hahn--Banach theorem we can find $\widetilde\Lambda\in (L^{p}(\Omega_T;\R^n))'$,
such that $\widetilde\Lambda|_{J(\VM(\Omega_T))}=\Lambda$ and $\|\widetilde\Lambda\|_{(L^{p}(\Omega_T))'}=\|\Lambda\|_{(J(\VM(\Omega_T)))'}$.
Therefore, by the Riesz representation Theorem in Lebesgue spaces, there exists $F\in L^{p'}(\Omega_T;\R^n)$
such that
\[
\widetilde\Lambda(H)
=
\int_{\Omega_T}F\cdot H\,dx\,dt
\qquad
\text{for every }H\in L^{p}(\Omega_T;\R^n) ,
\]
and
\[
\begin{aligned}
\|F\|_{L^{p'}(\Omega_T)}=
\|\widetilde\Lambda\|_{L^{p'}(\Omega_T)}
=
\|\Lambda\|_{(J(\VM(\Omega_T)))'}
=
\|\mathcal G\|_{\VM'}.
\end{aligned}
\]

Taking $H=\nabla\varphi=J\varphi$ with $\varphi\in V_M(\Omega_T)$, and using
\eqref{eq:Lambda-on-gradient-range}, we obtain
\[
-\langle\mathcal G,\varphi\rangle_{\VM',\VM}
=\Tilde{\Lambda}(\nabla\varphi)=
\int_{\Omega_T}F\cdot\nabla\varphi\,dx\,dt,
\]
which proves \eqref{eq:divergence-representation} and shows that $\mathcal G=\Div F$ in the sense of distributions. It remains to prove the minimizing property: the field constructed above
satisfies $\|F\|_{L^{p'}(\Omega_T)}=
\|\mathcal G\|_{\VM'}$, so the infimum in \eqref{eq:negative-norm-representation} is at most
$\|\mathcal G\|_{\VM'}$. Conversely, let
$H\in L^{p'}(\Omega_T;\R^n)$ be any vector field such that $\mathcal G=\Div H$. Then, for every $\varphi\in\VM(\Omega_T)$, H\"older's inequality gives
\[
\begin{aligned}
|\langle\mathcal G,\varphi\rangle|
=
\left|
\int_{\Omega_T}H\cdot\nabla\varphi\,dx\,dt
\right|
\le
\|H\|_{L^{p'}(\Omega_T)}
\|\nabla\varphi\|_{L^p(\Omega_T)}
=
\|H\|_{L^{p'}(\Omega_T)}
\|\varphi\|_{\VM}.
\end{aligned}
\]
Then, we obtain $\|\mathcal G\|_{\VM'}
\le
\|H\|_{L^{p'}(\Omega_T)}$, and
taking the infimum over all admissible fields $H$ concludes the proof of
\eqref{eq:negative-norm-representation}.
\end{proof}

\subsection{The Moring--Scheven space and capacity}

The authors of \cite{MS} consider the already introduced
parabolic energy space 
\begin{equation*}
\WM(\Omega_T)
:=
\left\{
 u\in \VM(\Omega_T)\cap L^\infty(0,T;L^2(\Omega)):
 \partial_tu\in\VM'(\Omega_T)
\right\},
\end{equation*}
endowed with the homogeneous norm \eqref{eq:intro-WM-norm}. The quantity we are going to minimize in \eqref{eq:capM-compact} and \eqref{eq:capM-extension}, see \cite[Definition~3.5 and Remark~3.6]{MS}, is the nonhomogeneous energy
\begin{equation*}
\EnergyM(u)
:=
\|\nabla u\|_{L^p(\Omega_T)}^p
+
\|u\|_{L^\infty(0,T;L^2(\Omega))}^2+
\|\partial_tu\|_{\VM'(\Omega_T)}^{p'}.
\end{equation*}
\begin{Corollary}
\label{cor:normalized-energy-bound}
Let $u\in W_M(\Omega_T)$, and let $F\in L^{p'}(\Omega_T;\mathbb R^n)$ be the norm-minimizing vector field provided by
Lemma~\ref{lem:divergence-representation}. If either $\mathcal E_M(u)\le1$ or $\|u\|_{W_M}\le1,$
then
\begin{equation*}
\int_{\Omega_T}
\left(
|\nabla u|^p+|F|^{p'}
\right)\,dx\,dt
\le1.
\end{equation*}
\end{Corollary}
\begin{proof}
Using that $\partial_tu=\Div F$ in the sense of distributions, and that $\|F\|_{L^{p'}(\Omega_T)}=\|\partial_tu\|_{V_M'}$, the statement follows by elementary inequalities.
\end{proof}
Let 
\begin{equation}\label{smoothspace}
\mathscr D
:=
\left.C_c^\infty(\Omega\times\R)\right|_{\Omega_T}.
\end{equation}
This space was first considered in \cite[Definition 2.10]{DPP}. Note that compactness is imposed in the spatial variable inside $\Omega$, while a
function in $\mathscr D$ is not required to vanish near $t=0$ or $t=T$.  For a compact set $K\Subset\Omega_T$, define
\begin{equation}\label{eq:capM-compact}
\capM(K;\Omega_T)
:=
\inf\left\{
\EnergyM(\varphi):
\varphi\in\mathscr D,\ 
\varphi\ge\chi_K
\text{ in }\Omega_T
\right\},
\end{equation}
where $\chi_K$ denotes the characteristic function of $K$. For an open set $U\subset\Omega_T$ and an arbitrary 
$E\subset\Omega_T$, set
\begin{equation}\label{eq:capM-extension}
\capM(U;\Omega_T)
:=
\sup_{K\Subset U}\capM(K;\Omega_T),
\qquad
\capM(E;\Omega_T)
:=
\inf_{\substack{U\supset E\\U\,\text{open}}}\capM(U;\Omega_T);
\end{equation}
we write simply $\capM(E)$ when the reference cylinder is clear. This is the
variational capacity of \cite[Definition~3.5]{MS}. Starting from an open set $U\subset \Omega_T$, it will be useful to define also
\begin{equation*}
\begin{aligned}
&\widehat{\operatorname{cap}}_M(U)
:=
\inf
\left\{
\|\varphi\|_{\WM(\Omega_T)}:
\varphi\in\WM(\Omega_T),\
\varphi\geq\chi_U\,\,
\text{a.e. in }\Omega_T
\right\}, \\
&\widetilde{\operatorname{cap}}_M(U)
:=
\inf
\left\{
\EnergyM(\varphi):
\varphi\in\WM(\Omega_T),\
\varphi\geq\chi_U
\,\,\text{a.e. in }\Omega_T
\right\},
\end{aligned}
\end{equation*}
with the same extension properties for Borel sets as in \eqref{eq:capM-extension}. 
\subsection{The Droniou--Porretta--Prignet space and capacity}
Let
\[
\VD:=W^{1,p}_0(\Omega)\cap L^2(\Omega),
\qquad
\VD':=(W^{1,p}_0(\Omega)\cap L^2(\Omega))'.
\]
The space considered by the authors of \cite{DPP} is
\begin{equation*}
\WD(\Omega_T)
:=
\left\{
 u\in L^p(0,T;\VD):
 \partial_tu\in L^{p'}(0,T;\VD')
\right\},
\end{equation*}
with norm
\[
\|u\|_{\WD}
:=
\|u\|_{L^p(0,T;\VD)}
+
\|\partial_tu\|_{L^{p'}(0,T;\VD')}.
\]
They also consider the space
\begin{equation*}
\WtD(\Omega_T)
:=
\left\{
 u\in L^p(0,T;W^{1,p}_0(\Omega))\cap L^\infty(0,T;L^2(\Omega)):
 \partial_tu\in L^{p'}(0,T;W^{-1,p'}(\Omega))
\right\}.
\end{equation*}
Thus, with the present notation
\begin{equation}\label{eq:WM-WtildeD}
\WtD(\Omega_T)=\WM(\Omega_T).
\end{equation}

\begin{Lemma}[Continuous embedding]
\label{lem:WM-into-WD}
The embedding $\WM(\Omega_T)\hookrightarrow\WD(\Omega_T)$ is continuous. More precisely,
\begin{equation}\label{eq:WM-WD-embedding}
\|u\|_{\WD}
\le
C(p,\Omega,T)\homnorm{u},
\qquad
\text{for every }u\in\WM(\Omega_T).
\end{equation}
\end{Lemma}

\begin{proof}
By Minkowski's inequality and the finiteness of the time interval,
\[
\begin{aligned}
\|u\|_{L^p(0,T;\VD)}
\le
\|u\|_{L^p(0,T;W^{1,p}(\Omega))}
+
\|u\|_{L^p(0,T;L^2(\Omega))}
\le
C\|\nabla u\|_{L^p(\Omega_T)}
+
T^{1/p}\|u\|_{L^\infty(0,T;L^2(\Omega))}.
\end{aligned}
\]
Moreover, the continuous embedding
$\VD\hookrightarrow W^{1,p}_0(\Omega)$ gives, by duality, $W^{-1,p'}(\Omega)\hookrightarrow\VD'$.
Hence
\[
\|\partial_tu\|_{L^{p'}(0,T;\VD')}
\le
C\|\partial_tu\|_{L^{p'}(0,T;W^{-1,p'}(\Omega))}.
\]
Combining the estimates proves \eqref{eq:WM-WD-embedding}.
\end{proof}

There, the authors introduce two different capacities. For an open set
$U\subset\Omega_T$,
\begin{equation}\label{eq:capD-open}
\capD(U)
:=
\inf\big\{
\|\varphi\|_{\WD}:
 \varphi\in\WD(\Omega_T),\ \varphi\ge\chi_U\text{ a.e. in }\Omega_T
\big\},
\end{equation}
and for a Borel set $B$,
\begin{equation}\label{eq:capD-Bor}
 \capD(B):=\inf\big\{\capD(U):B\subset U,\ U\text{ open}\big\}.   
\end{equation}
They also define, for compact $K\Subset\Omega_T$,
\begin{equation*}
\CAPD(K)
:=
\inf\big\{
\|\varphi\|_{\WD}:
\varphi\in\mathscr D,\ 
\varphi\ge\chi_K\text{ in }\Omega_T
\big\},
\end{equation*}
with the corresponding inner--outer extension to open and Borel sets, as in \eqref{eq:capM-extension}. For our purposes, we will use only the first notion of parabolic capacity; however, in the next result we shall present both for the sake of completeness. 

\subsection{Equivalence of polar sets}

The capacities just introduced should not be identified pointwise. We only
need the equivalence of their negligible sets. 

\begin{Proposition}[Equivalence at the level of polar sets]
\label{prop:capacity-null-equivalence}
For every Borel set $E\subset\Omega_T$,
\begin{equation}\label{eq:capacity-null-equivalence}
\widehat{\operatorname{cap}}_M(E)=0 \quad \Longleftrightarrow\quad
\capM(E)=0
\quad\Longleftrightarrow\quad
\capD(E)=0
\quad\Longleftrightarrow\quad
\CAPD(E)=0.
\end{equation}
\end{Proposition}

\begin{proof}
The equivalence
\[
\capD(E)=0\quad\Longleftrightarrow\quad\CAPD(E)=0
\]
is \cite[Proposition~2.14]{DPP}. By
\cite[Remark~2.18]{DPP}, replacing $\WD$ by the space $\widetilde{W}_D$ in the
definitions \eqref{eq:capD-open} and \eqref{eq:capD-Bor} does not change the polar sets.
In view of \eqref{eq:WM-WtildeD}, this shows the equivalence of polar sets between $\capD$ and  $\widehat{\operatorname{cap}}_M$.
It remains to compare this latter capacity with $\capM$.
In \cite[Lemmas~A.3,~A.6]{MS} it is shown that
\begin{equation}\label{lemma36}
  \widetilde{\operatorname{cap}}_M(E)\le   \operatorname{cap}_M(E)\le c(n,p)\, \widetilde{\operatorname{cap}}_M(E).
\end{equation}
Moreover, if
\[
A(v):=\|\nabla v\|_{L^p(\Omega_T)},
\qquad
B(v):=\|\partial_tv\|_{\VM'},
\qquad
C(v):=\|v\|_{L^\infty(0,T;L^2(\Omega))},
\]
then
\[
\EnergyM(v)=A(v)^p+B(v)^{p'}+C(v)^2,
\qquad
\homnorm{v}=A(v)+B(v)+C(v),
\]
making obvious the equivalence of negligible sets between $\widetilde{\operatorname{cap}}_M$ and $\widehat{\operatorname{cap}}_M$. This fact, combined with \eqref{lemma36}, concludes the proof.
\end{proof}

\subsection{Quasi-continuous representatives}

A function is $\capD$-quasi-continuous if, for every $\varepsilon>0$, there
exists an open set $U_\varepsilon\subset\Omega_T$ with
$\capD(U_\varepsilon)<\varepsilon$ such that the function is continuous on
$\Omega_T\setminus U_\varepsilon$. By
Proposition~\ref{prop:capacity-null-equivalence}, the expression ``quasi
everywhere'' is independent of the chosen capacity among those in
\eqref{eq:capacity-null-equivalence}.

\begin{Theorem}
\label{thm:DPP-representative}
For every $u\in\WD(\Omega_T)$ there exists a function $\widetilde u$ such
that
\[
\widetilde u=u
\qquad
\mathcal L^{n+1}\text{-a.e. in }\Omega_T,
\]
$\widetilde u$ is $\capD$-quasi-continuous, and it is unique
$\capD$-quasi everywhere among the quasi-continuous representatives of the
same Lebesgue equivalence class. Moreover, if $u_j\to u$ strongly in
$\WD(\Omega_T)$, then there exists a subsequence such that
\[
\widetilde{u}_{j_k}\to\widetilde u
\qquad
\capD\text{-quasi everywhere in }\Omega_T.
\]
\end{Theorem}

\begin{proof}
This is the content of \cite[Lemmas~2.20--2.21]{DPP}.
\end{proof}

Since $\WM\hookrightarrow\WD$ continuously and the polar sets coincide, every
$u\in\WM$ has a representative uniquely determined $\capM$-quasi
everywhere.

\begin{Lemma}[Linearity of quasi-continuous representatives]
\label{lem:linear-representatives}
Let $u,v\in\WD(\Omega_T)$ and $\alpha\in\R$. Then
\[
\widetilde{u+v}=\widetilde u+\widetilde v,
\qquad
\widetilde{\alpha u}=\alpha\widetilde u,
\]
$\capD$-quasi everywhere in $\Omega_T$.
\end{Lemma}

\begin{proof}
Fix $\varepsilon>0$. There are open sets $A_\varepsilon$ and
$B_\varepsilon$ with
\[
\capD(A_\varepsilon)<\frac{\varepsilon}{2},
\qquad
\capD(B_\varepsilon)<\frac{\varepsilon}{2},
\]
such that $\widetilde u$ and $\widetilde v$ are continuous outside the
corresponding sets. By \cite[Proposition 2.8]{DPP}, we have $\capD(A_\varepsilon\cup B_\varepsilon)<\varepsilon$,
and $\widetilde u+\widetilde v$ is continuous on the complement. Hence, $\widetilde u+\widetilde v$ is $\operatorname{cap}_D$-quasi-continuous. Moreover, 
\begin{equation*}
   \widetilde u+\widetilde v=u+v,\quad \mathcal L^{n+1}\text{ almost everywhere},
\end{equation*}
and then $\widetilde u+\widetilde v$ is a quasi-continuous representative of the
Lebesgue class $u+v$. Uniqueness in
Theorem~\ref{thm:DPP-representative} gives the first identity and the proof for
scalar multiplication is identical.
\end{proof}

\section{Preliminary results}\label{sec:preliminary}
In this section we fix the notation we will use throughout this article and present some preliminary results leading up to the proof of Theorem \ref{thm:intrinsic-duality}.\\

For an integrable function $v$ and a measurable set $A\subset\R^n$ with positive and finite Lebesgue measure, denoted by $\mathcal{L}^n(A)$ or $|A|$, we write
\[
(v)_A:=\mean_A v\,dx:=\frac{1}{|A|}\int_Av\,dx.
\]
Positive constants, possibly varying from line to line, will be denoted
by $c$, $\Tilde{c}$, $C$, $C_0$, $C_0,\,C_1$ or $C_2$. Relevant dependencies will be indicated between parentheses when needed. 

\medskip 

The first result we need is an extension lemma for functions belonging to the class $W_M(\Omega_T)$ and the corresponding vector field of Lemma \ref{lem:divergence-representation}. For $R>0$, we set
\[
I_R:=(-R^p,T+R^p).
\]

\begin{Lemma}
\label{lem:time-reflection}
Let $R>0$ and let $u\in\WM(\Omega_T)$. There exists an extension
$\overline u$ to $\Omega\times I_R$ such that $\overline u=u$ on $\Omega_T$,
\[
\overline u\in
L^p(I_R;W_0^{1,p}(\Omega))
\cap L^\infty(I_R;L^2(\Omega)),
\qquad
\partial_t\overline u
\in L^{p'}(I_R;W^{-1,p'}(\Omega)),
\]
and
\begin{equation}\label{eq:time-reflection-bound}
\|\overline u\|_{L^p(I_R;W^{1,p}(\Omega))}
+
\|\overline u\|_{L^\infty(I_R;L^2(\Omega))}
+
\|\partial_t\overline u\|_{L^{p'}(I_R;W^{-1,p'}(\Omega))}
\le C(R,T)\homnorm{u}.
\end{equation}
If $F\in L^{p'}(\Omega_T;\R^n)$ satisfies
$\partial_tu=\Div F$, then the extension can be chosen together with a vector field $\overline F$ satisfying
\[
\partial_t\overline u=\Div\overline F
\quad\text{in }\Omega\times I_R,
\qquad
\|\overline F\|_{L^{p'}(\Omega\times I_R)}
\le C(R,T)\|F\|_{L^{p'}(\Omega_T)}.
\]
\end{Lemma}

\begin{proof}
By Lemma~\ref{lem:WM-into-WD} and the Lions-Magenes time-continuity theorem (see \cite[Remark~2.2]{DPP}),
$u$ has a representative in the space $C([0,T];L^2(\Omega))$. Reflect $u$ evenly first
across $t=0$ and then across $t=T$. At each gluing time the two $L^2$ traces
coincide, so the gluing lemma
\cite[Lemma~2.21]{MS} shows that the reflected function remains in the same
parabolic space. Repeating the construction finitely many times gives an
extension to $I_R$. The number of copies depends only on $R$ and $T$, which
yields \eqref{eq:time-reflection-bound}.

On a reflected copy of $(0,T)$, define $\overline F$ by reflecting $F$ and
multiplying it by the sign of the time reflection. For example, on $(-T,0)$, $\overline u(t)=u(-t)$, $\overline F(t)=-F(-t)$. Then $\partial_t\overline u=\Div\overline F$. Observe that, since $u \in C([0,T]; L^2(\Omega))$, the matching $L^2$-traces prevent the formation of Dirac masses in the weak temporal derivative, ensuring that the equation holds in the distributional sense across the gluing times. The same rule applies at each
successive reflection, and the norm estimate follows from the finite number of
copies.
\end{proof}
The following results include two Poincaré-type inequalities and will be useful for the estimates in the next section. In fact, it will be necessary to distinguish between cylinders entirely contained within the initial domain and cases in which they extend beyond it.  
\begin{Lemma}[Interior parabolic Poincar\'e inequality]
\label{lem:parabolic-poincare}
Let $z=(x,t)$ and let
$Q_{r,r^p}(z)\subset\Omega\times I_R$. Suppose that
$v\in L^p(t-r^p,t+r^p;W^{1,p}(B_r(x)))$
and that $\partial_tv=\Div F$ in the sense of distributions in
$Q_{r,r^p}(z)$, with $F\in L^1(Q_{r,r^p}(z);\R^n)$.
Then
\begin{equation}\label{eq:parabolic-poincare}
\begin{aligned}
\mean_{Q_{r,r^p}(z)}
|v-(v)_{Q_{r,r^p}(z)}|^p\,dx\,dt
\le
Cr^p\mean_{Q_{r,r^p}(z)}|\nabla v|^p\,dx\,dt
+
C\left(
 r^{p-1}\mean_{Q_{r,r^p}(z)}|F|\,dx\,dt
\right)^p.
\end{aligned}
\end{equation}
\end{Lemma}

\begin{proof}
The proof follows that of \cite[Corollary~5.2]{MS}, after translating the notation
of that paper into $Q_{r,r^p}(z)$. Indeed, although \cite[Corollary~5.2]{MS} is stated for functions with zero trace on the boundary of the spatial ball,
the same boundary condition is not used. The only spatial ingredient is the usual averaged Poincaré's inequality.
\end{proof}
\begin{Lemma}[Boundary Poincar\'e inequality]
\label{lem:boundary-poincare}
There exists
$C=C(n,p,\Omega)>0$ such that, if $x\in\Omega,\,r>0$ and
\begin{equation*}
\dist(x,\partial\Omega):=\underset{y\in\partial\Omega}{\inf}\,|x-y|\le r,
\end{equation*}
then
\begin{equation}\label{eq:boundary-poincare}
\int_{B_r(x)\cap\Omega}|v|^p\,d\xi
\le
Cr^p
\int_{B_{5r/4(x)}\cap\Omega}|\nabla v|^p\,d\xi
\end{equation}
for every $v\in C_c^\infty(\Omega)$.
\end{Lemma}
\begin{proof}
Let $\widehat v\in C_c^\infty(\mathbb R^n)$ denote the zero
extension of $v$ outside $\Omega$. Choose $y\in\partial\Omega$ such that
\[
|x-y|
=
\dist(x,\partial\Omega)
\leq r,
\]
and set
\[
A:=B_{r/4}(y)\,\,
\subset\,\,
B:=B_{5r/4}(x).
\]
Since $\Omega$ is Lipschitz, its complement
$\mathbb R^n\setminus\Omega$ is uniformly $p$-thick (see Remark \ref{rem:weaker-boundary-assumption}). Hence, the boundary Poincar\'e inequality
\cite[Theorem~3.7]{KK}, applied to the ball $A$, gives 
\begin{equation}
\label{eq:KK-boundary-Poincare}
\|\widehat v\|_{L^p(A)}
\leq
Cr\|\nabla\widehat v\|_{L^p(A)}.
\end{equation}
By the standard Poincar\'e inequality on $B$, we have
\begin{equation}
\label{eq:standard-Poincare-B}
\|\widehat v-(\widehat v)_B\|_{L^p(B)}
\leq
Cr\|\nabla\widehat v\|_{L^p(B)}.
\end{equation}
Moreover,
\[
\begin{aligned}
|(\widehat v)_B-(\widehat v)_A|
=
\left|
\mean_A
\bigl((\widehat v)_B-\widehat v\bigr)\,d\xi
\right|
\leq
|A|^{-1/p}
\|\widehat v-(\widehat v)_B\|_{L^p(A)}
\leq
|A|^{-1/p}
\|\widehat v-(\widehat v)_B\|_{L^p(B)}.
\end{aligned}
\]
Also, by H\"older's inequality,
\[
|(\widehat v)_A|
\leq
|A|^{-1/p}
\|\widehat v\|_{L^p(A)}.
\]
Consequently,
\[
\begin{aligned}
|B|^{1/p}|(\widehat v)_B|
\leq
|B|^{1/p}
\left(
|(\widehat v)_B-(\widehat v)_A|
+
|(\widehat v)_A|
\right)
\leq
\left(\frac{|B|}{|A|}\right)^{1/p}
\left(
\|\widehat v-(\widehat v)_B\|_{L^p(B)}
+
\|\widehat v\|_{L^p(A)}
\right).
\end{aligned}
\]
Since
\[
\frac{|B|}{|A|}
=
\frac{|B_{5r/4}(x)|}{|B_{r/4}(y)|}
=
5^n,
\]
the preceding estimate and \eqref{eq:standard-Poincare-B} imply
\begin{equation}
\label{eq:average-on-B}
|B|^{1/p}|(\widehat v)_B|
\leq
Cr\|\nabla\widehat v\|_{L^p(B)}
+
C\|\widehat v\|_{L^p(A)}.
\end{equation}
Using Minkowski's inequality, once more \eqref{eq:standard-Poincare-B} and then
\eqref{eq:average-on-B}, we obtain
\[
\begin{aligned}
\|\widehat v\|_{L^p(B)}
\leq
\|\widehat v-(\widehat v)_B\|_{L^p(B)}
+
|B|^{1/p}|(\widehat v)_B|
\leq
Cr\|\nabla\widehat v\|_{L^p(B)}
+
C\|\widehat v\|_{L^p(A)}.
\end{aligned}
\]
The boundary estimate \eqref{eq:KK-boundary-Poincare}, together with
$A\subset B$, therefore yields
\[
\begin{aligned}
\|\widehat v\|_{L^p(B)}
\leq
Cr\|\nabla\widehat v\|_{L^p(B)}
+
Cr\|\nabla\widehat v\|_{L^p(A)}
\leq
Cr\|\nabla\widehat v\|_{L^p(B)}.
\end{aligned}
\]
Raising this inequality to the power $p$, we get
\begin{equation}
\label{eq:Poincare-on-B}
\int_B|\widehat v|^p\,d\xi
\leq
Cr^p
\int_B|\nabla\widehat v|^p\,d\xi.
\end{equation}
Finally, $B_r(x)\subset B$ and, by the definition of the zero extension,
\[
\widehat v=v
\quad\text{a.e. in }\Omega,
\qquad
\nabla\widehat v
=
\chi_\Omega\nabla v
\quad\text{a.e. in }\mathbb R^n.
\]
It follows from \eqref{eq:Poincare-on-B} that
\[
\begin{aligned}
\int_{B_r(x)\cap\Omega}|v|^p\,d\xi
\leq
\int_B|\widehat v|^p\,d\xi
\leq
Cr^p
\int_B|\nabla\widehat v|^p\,d\xi
=
Cr^p
\int_{B_{5r/4}(x)\cap\Omega}|\nabla v|^p\,d\xi.
\end{aligned}
\]
\end{proof}
The next result is a density lemma for the space $W_M(\Omega_T)$. Since the proof is rather technical, we prefer to include it in the final appendix and present only the statement here.
\begin{Lemma}[Smooth density]
\label{lem:smooth-density-WM}
Let $\Omega$ be a bounded Lipschitz open set and let $0<T<\infty$. Then, the class $\mathscr D$ introduced in \eqref{smoothspace} is dense in $\WM(\Omega_T)$ with respect to the homogeneous norm. More
precisely, for every $u\in\WM(\Omega_T)$ there exists a sequence
$\varphi_j\in\mathscr D$ such that
\[
\varphi_j\to u
\,\,\text{in }L^p(0,T;W^{1,p}_0(\Omega)),\quad
\varphi_j\to u
\,\,\text{in }L^\infty(0,T;L^2(\Omega)),\quad
\partial_t\varphi_j\to\partial_tu
\,\,\text{in }L^{p'}(0,T;W^{-1,p'}(\Omega)).
\]
\end{Lemma}
\subsection*{Morrey conditions and parabolic Hausdorff measure}
\begin{Remark}[Centers and small radii]\label{rem:centers}
Requiring \eqref{eq:intrinsic-morrey} only for centers in $\Omega_T$ is
equivalent, up to a fixed multiplicative constant, to requiring it for all
centers in $\R^{n+1}$. Indeed, let
$z_0\in\R^{n+1}$. If $Q_{r,r^q}(z_0)\cap\Omega_T=\emptyset$, there is nothing to prove. Otherwise choose
$w\in Q_{r,r^q}(z_0)\cap\Omega_T$. For every
$y\in Q_{r,r^q}(z_0)$,
\[
|y_x-w_x|<2r,
\qquad
|y_t-w_t|<2r^q,
\]
implying $Q_{r,r^q}(z_0)
\subset
Q_{2r,(2r)^q}(w)$. Hence, a centered Morrey estimate gives
\[
|\mu|\bigl(Q_{r,r^q}(z_0)\cap\Omega_T\bigr)
\le
M 2^{n+q-\vartheta}r^{n+q-\vartheta}.
\]
Moreover, it is enough to work with small radii $0<r\le1$. Indeed, suppose that an estimate
\[
|\mu|\bigl(Q_{r,r^q}(z)\cap\Omega_T\bigr)
\le Mr^s
\]
holds for $0<r\le1$, with $s\ge 0$. Then, for $r>1$,
\[
|\mu|\bigl(Q_{r,r^q}(z)\cap\Omega_T\bigr)
\le|\mu|(\Omega_T)
\le|\mu|(\Omega_T)r^s.
\]
Hence the estimate holds for every $r>0$ after replacing $M$ with
$\max\{M,|\mu|(\Omega_T)\}$.
\end{Remark}
In the present work, the relevant range in which condition \eqref{eq:intrinsic-morrey} holds is $0\le \vartheta\le n+q$. Indeed, if $\vartheta<0$, covering $\Omega_T$ by at most
$Cr^{-(n+q)}$ cylinders of radius $r$ and using
\eqref{eq:intrinsic-morrey} give
\[
|\mu|(\Omega_T)
\leq
CMr^{-\vartheta}
\to 0
\qquad\text{as }r\to 0^+.
\]
If $\vartheta>n+q$, a sufficiently large cylinder centered in
$\Omega_T$ contains the whole domain, and hence
\[
|\mu|(\Omega_T)
\leq
Mr^{n+q-\vartheta}
\to 0
\qquad\text{as }r\to\infty.
\]
The endpoint
$\vartheta=0$ corresponds to volume-type growth, whereas
$\vartheta=n+q$ gives a radius-independent bound and allows atoms.
In particular, the restriction $\vartheta>1$, frequently imposed in
nonlinear regularity theory, is not needed for the present
measure-theoretic duality problem.

Finally, we present a result that relates the parabolic capacity of \cite{MS} to the Hausdorff measure associated with the metric $d_q$, introduced in \eqref{eq:dp}. This result is crucial for providing a good definition of the functional associated with the measure $\mu$ and acting on the space $W_M(\Omega_T)$. 
\begin{Definition}[$q$-parabolic Hausdorff measure]\label{def:q-parabolic-Hausdorff-measure}
Let $E\subset\mathbb R^{n+1}$, $q>1,\,s\geq0,\,\delta>0$ and set

\[
\mathcal H^{s}_{q,\delta}(E)
:=
\inf
\biggl\{
\sum_{i=1}^{\infty}r_i^s:
E\subset
\bigcup_{i=1}^{\infty}Q_{r_i,r_i^q}(z_i),
\quad
0<r_i<\delta
\biggr\};
\]
the $s$-dimensional $q$-parabolic Hausdorff measure of $E$ is defined as
\[
\mathcal H_q^s(E)
:=
\lim_{\delta\to0^+}
\mathcal H^{s}_{q,\delta}(E)
=
\sup_{\delta>0}\mathcal H^{s}_{q,\delta}(E).
\]
\end{Definition}
\begin{Definition}\label{HausdorffDim}
We define the $q$-parabolic Hausdorff dimension of $E\subset\R^{n+1}$ as
\[
\dim^{\mathcal P}_{\mathcal H,q}(E)
:=
\inf\left\{
0\le s<\infty:\mathcal H_q^s(E)=0
\right\}.
\]
\end{Definition}
We again stress that parabolic Hausdorff measures built on peculiar values of $q$ play an essential role in the study of fine properties of solutions to equations as in \eqref{eq:intro-parabolic-problem}, see \cite{BESS}.
\begin{Remark}
The definition of the Hausdorff parabolic measure given in \cite[Section 5]{MS} differs slightly from the one presented here, in terms of the permissible coverings of the set. Nevertheless, the two measures have
the same null sets and define the same Hausdorff dimension. Indeed, let \(\mathcal H_{d_q}^s\) denote the Hausdorff measure associated with the metric \(d_q\), defined by coverings with arbitrary sets and using their \(d_q\)-diameters. Since
\[
Q_{r,r^q}(z)=B_{d_q}(z,r)
\qquad\text{and}\qquad
\operatorname{diam}_{d_q}\bigl(Q_{r,r^q}(z)\bigr)=2r,
\]
the cylindrical measure introduced above satisfies
\[
2^{-s}\mathcal H_{d_q}^s(E)
\leq
\mathcal H_q^s(E)
\leq
\mathcal H_{d_q}^s(E)
\]
for every \(E\subset\mathbb R^{n+1}\).  In
particular, when \(q=p\), the capacity--Hausdorff comparisons of
\cite[Theorem~1.2]{MS} apply equivalently to
\(\mathcal H_p^s\).
\end{Remark}
\begin{Proposition}[Diffuseness with respect to capacity]
\label{prop:diffuse}
Let $\mu$ be a finite signed Radon measure satisfying
\eqref{eq:intrinsic-morrey} with $q=p$ and $\vartheta<p$. Then
\[
|\mu|\ll\capM
\]
in the sense that, if $E\subset\Omega_T$ is Borel and $\capM(E)=0$, then
$|\mu|(E)=0$. Consequently,
$\mu^+$, $\mu^-$, and $\mu$ also vanish on every
$\capM$-polar Borel set.
\end{Proposition}
\begin{proof}
Set $s:=n+p-\vartheta>n$. Fix $\delta>0$ and let
\[
E\subset\bigcup_{i}Q_{r_i,r_i^p}(z_i),
\qquad
0<r_i<\delta,
\]
be a countable covering. By Remark \ref{rem:centers},
we can choose $w_i\in E\cap Q_{r_i,r_i^p}(z_i)$ such that 
\[
Q_{r_i,r_i^p}(z_i)
\subset
Q_{2r_i,(2r_i)^p}(w_i).
\]
The centers $w_i$ belong to $\Omega_T$, so the Morrey estimate gives
\[
\begin{aligned}
|\mu|(E)
\le
\sum_i|\mu|\bigl(Q_{2r_i,(2r_i)^p}(w_i)\cap\Omega_T\bigr)
\le
2^sM\sum_i r_i^s;
\end{aligned}
\]
taking the infimum over all coverings and letting $\delta\to0^+$ yields
\begin{equation}\label{eq:mu-Hausdorff}
0\le |\mu|(E)\le 2^sM\calH_p^s(E).
\end{equation}
Since the complement of a Lipschitz domain is uniformly $p$-thick,
the hypotheses of \cite[Theorem~1.2]{MS} are satisfied, implying that
\[
\capM(E)=0
\quad\implies\quad
\calH_p^\sigma(E)=0,
\qquad\text{for every }\sigma>n.
\]
Applying this with $\sigma=s$ and using \eqref{eq:mu-Hausdorff} proves the
claim. Finally, for every Borel set $E$ one has 
$0\leq\mu^\pm(E)\leq|\mu|(E)$, and then $|\mu(E)|=0$ as well. Hence $\mu^+$, $\mu^-$, and $\mu$ also vanish on all $\capM$-polar Borel sets.
\end{proof}

\section{Proof of the main results}\label{sec:duality}
With the preliminary results from the previous section in hand, we are ready to prove the main result of this work. The proof combines parabolic capacity theory with a quantitative covering argument. We first consider the case $q=p$, namely we assume that \eqref{eq:intrinsic-morrey} holds on the $(r,r^p)$-cylinders with $\vartheta<p$. Once the duality result has been established in this setting, the corresponding conclusion on $(r,r^q)$-cylinders follows immediately by comparison with parabolic ones. 

\medskip

\textbf{The strategy of the proof for the case $q=p$.} 
We first consider smooth functions $\varphi$ with $\homnorm{\varphi}\le1$ and write $\partial_t\varphi=\operatorname{div}F$, with $F\in L^{p'}$: a suitable reflection in time preserves this identity and allows us to work on complete parabolic cylinders, including
those crossing the temporal endpoints.

A dyadic argument based on the interior Poincar\'e inequality,
completed by the boundary inequality when necessary, associates
with each point of a superlevel set a cylinder carrying
a sufficiently large amount of energy.
The condition $\vartheta<p$ ensures convergence of the
geometric series arising in this argument.
A Vitali covering and the Morrey bound then yield
\[
|\mu|\bigl(\{|\varphi|>\lambda\}\bigr)
\le C\lambda^{-\min\{p,p'\}},
\qquad \lambda\ge1.
\]

Since $\min\{p,p'\}>1$, integration over the level sets gives the desired $L^1(d|\mu|)$ estimate, and homogeneity removes the normalization. Finally, strong approximation in $\WM$, together with quasi-continuity and the fact that $|\mu|$ vanishes on polar sets, yields convergence almost everywhere with respect to $|\mu|$ along a subsequence. Fatou's lemma extends the estimate to all $u\in\WM(\Omega_T)$.

\medskip

\begin{proof}[Proof of Theorem \ref{thm:intrinsic-duality} for the case $q=p$]
We divide the proof into three steps.

\medskip

\noindent\textbf{Step 1: estimate for smooth functions.}
Let $\varphi\in\mathscr D$ and assume first that $\homnorm{\varphi}\le1$. By Lemma~\ref{lem:divergence-representation}, there exists
$F\in L^{p'}(\Omega_T;\R^n)$ such that $\partial_t\varphi=\Div F$ in the sense of distributions, with $\|F\|_{L^{p'}(\Omega_T)}
=\|\partial_t\varphi\|_{\VM'}$. 
In view of Lemma \ref{cor:normalized-energy-bound} we have
\begin{equation*}
\int_{\Omega_T}\left(|\nabla\varphi|^p+|F|^{p'}\right)\,dx\,dt\leq 1;
\end{equation*}
hence, applying Hölder's and the slice-wise Poincar\'e inequalities also gives
\begin{equation*}
\int_{\Omega_T}|\varphi|\,dx\,dt\leq |\Omega_T|^{1/p'}C_P(\Omega)\|\nabla\varphi\|_{L^p(\Omega_T)}\le C_0,
\end{equation*}
where $C_P(\Omega)>0$ denotes the Poincaré constant of $\Omega$. We now extend the function on the time variable, by applying Lemma \ref{lem:time-reflection}.
For a fixed $r\ge 1$, the number $N(r)$ of reflected copies of $(0,T)$ needed to cover $I_{2r}:=(-(2r)^p,T+(2r)^p)$ satisfies $N(r)\le Cr^p$. Therefore, the reflected extension satisfies
\begin{equation}\label{reflect1}
\int_{\Omega\times I_{2r}}|\overline\varphi|\,dx\,dt
\le C_1r^p,
\end{equation}
for a new constant $C_1>0$ depending only on $p,\Omega,T$. Since $|Q_{r,r^p}|=c_nr^{n+p}$,
we may fix a structural radius $r_0=r_0(n,p,\Omega,T)\geq1$ so large that
\begin{equation}\label{use2}
\frac{C_1r_0^p}{|Q_{r_0,r_0^p}|}
=
\frac{C_1}{c_nr_0^n}
\leq\frac14.
\end{equation}
We now apply the reflection lemma \ref{lem:time-reflection} to $\varphi$ and $F$, with $R=2r_0$ and set
\begin{equation*}\label{newdom}
I_*:=
\bigl(-(2r_0)^p,T+(2r_0)^p\bigr),
\qquad
\Omega_*:=\Omega\times I_*;
\end{equation*}
in particular, $\overline\varphi\in W_M(\Omega_*)$, $F\in L^{p'}(\Omega_*;\R^n)$ and
$\partial_t\overline\varphi=\Div\overline F$ continues to hold in the sense of distributions. Moreover, each spatial slice
\(\overline\varphi(\cdot,t)\) belongs to \(C_c^\infty(\Omega)\), and by \eqref{reflect1}
\begin{equation}\label{eq:extended-G-bound}
\int_{\Omega_*}|\overline{\varphi}|\,dx\,dt\le C_1r_0^p,
\qquad\int_{\Omega_*} |\overline{G}|\,dx\,dt:=\int_{\Omega_*}
\left(|\nabla\overline\varphi|^p+|\overline F|^{p'}\right)
\,dx\,dt
\le C_2,
\end{equation}
where $C_1=C_1(p,\Omega,T),\,C_2=C_2(T,r_0)>0$ are  structural constants. If $z=(x,t)\in\Omega_T$ satisfies
$\dist(x,\partial\Omega)\ge r_0$, then
$Q_{r_0,r_0^p}(z)\subset\Omega_*$ and, by
\eqref{use2} and \eqref{eq:extended-G-bound} we get
\begin{equation}\label{eq:large-average}
\mean_{Q_{r_0,r_0^p}(z)}|\overline\varphi|\,dx\,dt\le \frac{1}{|Q_{r_0,r_0^p}|}\int_{\Omega_*}|\overline\varphi|\,dx\,dt\le \frac{C_1r_0^p}{|Q_{r_0,r_0^p}|}
\le\frac14.
\end{equation}
A temporal extension is necessary because a parabolic cylinder
centered at $z=(x,t)\in\Omega_T$ need not be contained in the original time interval. 
Replacing the cylinder by
$Q_{r,r^p}(z)\cap\Omega_T$ would not permit a direct application of
the interior Poincar\'e inequality \eqref{eq:parabolic-poincare}, which is formulated on complete
cylinders. Moreover, extending $\varphi$ by zero in time is not suitable because 
\[
\partial_t
\bigl(
\chi_{(0,T)}\varphi
\bigr)
=
\chi_{(0,T)}\partial_t\varphi
+
\varphi(\cdot,0)\delta_{\{t=0\}}
-
\varphi(\cdot,T)\delta_{\{t=T\}},
\]
so that artificial Dirac masses appear at the temporal endpoints.
Extending $F$ by zero would not produce the corresponding terms,
since the divergence acts only in the spatial variables.\\

Fix $\lambda\ge1$ and $z=(x,t)\in E_\lambda(\varphi):=\{z\in\Omega_T:\varphi(z)>\lambda\}$, and set
\[
d_z:=\dist(x,\partial\Omega)
\qquad
\varrho_z:=\min\{r_0,d_z\}.
\]
We claim that there exists
$r_z\in(0,5r_0/4]$ such that
\begin{equation}\label{eq:bad-radius}
\frac{1}{|Q_{r_z,r_z^p}|}
\int_{Q_{r_z,r_z^p}(z)\cap\Omega_*}
\overline G\,dx\,dt
>
c_0\min\{\lambda^p,\lambda^{p'}\}\, r_z^{-\vartheta},
\end{equation}
where $c_0>0$ is chosen below, sufficiently small in terms of the structural data. Suppose first that
\begin{equation}\label{eq:no-interior-bad-radius}
\mean_{Q_{r,r^p}(z)}\overline G\,dx\,dt
\le c_0\min\{\lambda^p,\lambda^{p'}\}\, r^{-\vartheta}
\qquad
\text{for every }0<r\le\varrho_z.
\end{equation}
Set $r_k:=2^{1-k}\varrho_z$ for $k\in\N$, and for $r\in(0,\varrho_z]$ choose $m\in\N$ such that $r_{m+1}<r\le r_m$. From now on, we set for brevity 
\begin{equation*}
 (\overline{\varphi})_r=  (\overline{\varphi})_{Q_{r,r^p}(z)},\quad 
 (\overline{\varphi})_{r_0}=(\overline{\varphi})_{Q_{r_0,r_0^p}(z)},\quad
 (\overline{\varphi})_z= (\overline{\varphi})_{Q_{\varrho_z,\varrho_z^p}(z)},\quad (\overline{\varphi})_k= (\overline{\varphi})_{Q_{r_k,r_k^p}(z)}.
\end{equation*}
 From the telescopic expansion
\begin{equation*}
    (\overline{\varphi})_{r}-(\overline{\varphi})_{z}=(\overline{\varphi})_{r}-(\overline{\varphi})_m+\sum_{k=1}^{m-1}\left[(\overline{\varphi})_{k+1}-(\overline{\varphi})_{k}\right],
\end{equation*}
we can use elementary estimates and Hölder's inequality to obtain
\begin{equation*}
  |(\overline{\varphi})_{r}-(\overline{\varphi})_{z}|\le C\sum_{k=1}^{m}
  \biggl(
  \mean_{Q_{r_k,r_k^p}(z)}
  |\overline{\varphi}-(\overline{\varphi})_{k}|^p\,dx\,dt
  \biggr)^{1/p}\le C\sum_{k=1}^{\infty}
\biggl(
\mean_{Q_{r_k,r_k^p}(z)}
|\overline\varphi-(\overline\varphi)_{k}|^p
\,dx\,dt
\biggr)^{1/p}.
 \end{equation*}
Letting $r\to 0^+$ and using the smoothness of $\overline{\varphi}$ at $z$, we get
\begin{equation}\label{eq:infinite-telescopic}
\begin{aligned}
|\overline{\varphi}(z)-(\overline\varphi)_{z}|
&\le
C\sum_{k=1}^{\infty}
\biggl(
\mean_{Q_{r_k,r_k^p}(z)}
|\overline\varphi-(\overline\varphi)_{k}|^p
\,dx\,dt
\biggr)^{1/p}.
\end{aligned}
\end{equation}
For $0<r\le\varrho_z$, the cylinder
$Q_{r,r^p}(z)$ is contained in $\Omega_*$.
Lemma~\ref{lem:parabolic-poincare}, H\"older's inequality, and
\eqref{eq:no-interior-bad-radius} hence yield
\begin{equation*}
\begin{aligned}
\biggl(
\mean_{Q_{r_k,r_k^p}(z)}
|\overline\varphi-(\overline\varphi)_{k}|^p
\,dx\,dt
\biggr)^{1/p}
&\le Cr_k\biggl(\mean_{Q_{r_k,r_k^p}(z)}|\nabla \overline\varphi|^p
\,dx\,dt\biggr)^{1/p}+ C {r_k}^{p-1}\biggl(\mean_{Q_{r_k,r_k^p}(z)}|F|^{p'}\,dx\,dt\biggr)^{1/p'} \\
&\le C c_0^{1/p}\min\left\{\lambda,\lambda^{1/(p-1)}\right\}
r_k^{(p-\vartheta)/p}
+
C c_0^{1/p'}\min\left\{\lambda^{p-1},\lambda\right\}
r_k^{(p-\vartheta)/p'} \\
&\le C\lambda\biggl[ c_0^{1/p}
r_k^{(p-\vartheta)/p}
+ c_0^{1/p'}
r_k^{(p-\vartheta)/p'}\biggr].
\end{aligned}
\end{equation*}
Inserting this bound in \eqref{eq:infinite-telescopic}, we observe that both dyadic series converge, because $\vartheta<p$ and $\rho_z\le r_0$.
At this point, $c_0$ can be
chosen small enough, independently of $\varphi,z$ and $\lambda$, so that 
\begin{equation}\label{postpoi}
|{\varphi}(z)-(\overline\varphi)_{z}|
\le\frac{\lambda}{4}.
\end{equation}
If $\varrho_z=r_0$, then every cylinder in the dyadic chain up to the
scale $r_0$ is spatially contained in $\Omega_*$, and by \eqref{eq:large-average} we get $\varphi(z)\le\lambda/2$,
contradicting $z\in E_\lambda(\varphi)$. Hence, in this case, some
radius $r_z\in(0,r_0]$ must satisfy \eqref{eq:bad-radius}. If $\varrho_z=d_z$, the dyadic argument leading to \eqref{postpoi} is still valid up to the terminal scale $d_z$, but
it cannot in general be continued from $d_z$ to $r_0$:  for $r>d_z$, the ball $B_r(x)$ crosses the spatial boundary, and
the identity $\partial_t\overline\varphi
=
\operatorname{div}\overline F$
is known only in $\Omega_*$. This is precisely where the boundary Poincar\'e inequality \eqref{eq:boundary-poincare} is needed. Combining \eqref{postpoi} with $\varphi(z)>\lambda$ implies
$\left|(\overline\varphi)_{z}\right|
>3\lambda/4$ and, by Jensen's inequality
\begin{equation}\label{eq:large-boundary-average}
\left(\frac{3\lambda}{4}\right)^p
<
\mean_{Q_{\varrho_z,\varrho_z^p}(z)}
|\overline\varphi|^p\,dx\,dt.
\end{equation}
For almost every $s\in(t-\varrho_z^p,t+\varrho_z^p)$,
Lemma~\ref{lem:boundary-poincare} applies to $\overline{\varphi}(\cdot,s)$, with
$r=\varrho_z$. Integrating in time and
using
\[
B_{5\varrho_z/4}(x)\times(t-\varrho_z^p,t+\varrho_z^p)
\subset
Q_{5\varrho_z/4,(5\varrho_z/4)^p}(z),
\]
we obtain
\[
\int_{Q_{\varrho_z,\varrho_z^p}(z)}
|\overline\varphi|^p\,dx\,dt
\le
C\varrho_z^p
\int_{Q_{5\varrho_z/4,(5\varrho_z/4)^p}(z)\cap\Omega_*}
|\nabla\overline\varphi|^p\,dx\,dt.
\]
Combining this with \eqref{eq:large-boundary-average} gives
\begin{equation}\label{eq:boundary-energy-lower}
\frac{1}{|Q_{5\varrho_z/4,(5\varrho_z/4)^p}|}
\int_{Q_{5\varrho_z/4,(5\varrho_z/4)^p}(z)\cap\Omega_*}
\overline G\,dx\,dt
\ge
c\lambda^p\varrho_z^{-p}.
\end{equation}
Since $\varrho_z\le r_0$ and $\vartheta<p$,
\[
\varrho_z^{-p}
\ge
r_0^{\vartheta-p}\varrho_z^{-\vartheta},
\]
and then inequality \eqref{eq:boundary-energy-lower} can be written as
\begin{equation}\label{eq:boundary-energy-lower1}
\frac{1}{|Q_{5\varrho_z/4,(5\varrho_z/4)^p}|}
\int_{Q_{5\varrho_z/4,(5\varrho_z/4)^p}(z)\cap\Omega_*}
\overline G\,dx\,dt\ge cr_0^{\vartheta-p}\min\{\lambda^p,\lambda^{p'}\}\varrho_z^{-\vartheta}.
\end{equation}
After decreasing $c_0$ once more, \eqref{eq:boundary-energy-lower1} implies
\eqref{eq:bad-radius} with $r_z=5\varrho_z/4=5d_z/4\le 5r_0/4$. The claim is proved in all
cases. 

It follows that every $z\in E_\lambda(\varphi)$ determines a
radius $0<r_z\leq5r_0/4$ for which
\begin{equation}\label{eq:radius-energy}
r_z^{n+p-\vartheta}
\leq
\frac{C}{\min\{\lambda^p,\lambda^{p'}\}}
\int_{Q_{r_z,r_z^p}(z)\cap\Omega_*}
\overline G\,dx\,dt.
\end{equation}
The family $\{Q_{r_z,r_z^p}(z)\}$ covers the open set $E_\lambda(\varphi)$. By the Vitali
covering lemma for the metric $\Dp$, there is a countable pairwise disjoint
subfamily $\{Q_i\}:=\{Q_{r_i,r_i^p}(z_i)\}_{i\in\N}$ such that
\[
E_\lambda(\varphi)\subset\bigcup_iQ_{5r_i,(5r_i)^p}(z_i),
\]
and where every center $z_i$ lies in $E_\lambda(\varphi)\subset\Omega_T$. Hence the Morrey condition \eqref{eq:intrinsic-morrey} for $q=p$,
\eqref{eq:radius-energy} and the pairwise disjointness of the cylinders yield
\begin{align*}
|\mu|(E_\lambda(\varphi))
\le
\sum_i
|\mu|\bigl(Q_{5r_i,(5r_i)^p}(z_i)\cap\Omega_T\bigr)
\le
&CM\sum_i r_i^{n+p-\vartheta}\\
\le
&\frac{C}{\min\{\lambda^p,\lambda^{p'}\}}
\sum_i
\int_{Q_i\cap\Omega_*}\overline G\,dx\,dt
\le
\frac{C}{\min\{\lambda^p,\lambda^{p'}\}}
\int_{\Omega_*}\overline G\,dx\,dt,
\end{align*}
where $C=C(n,p,\Omega,T,\vartheta,M)$.
Notice that the dilated cylinders need not be contained in $\Omega_*$: they are used only in the Morrey estimate. 
By repeating the same argument with $-\varphi$ and $- F$ in place of
$\varphi$ and $ F$, we obtain
\begin{equation}\label{eq:absolute-level-estimate}
|\mu|(E_\lambda(|\varphi|))=|\mu|(\{|\varphi|>\lambda\})
\le
\frac{C}{\min\{\lambda^p,\lambda^{p'}\}}
\int_{\Omega_*}\overline G\,dx\,dt.
\end{equation}
By Cavalieri's principle, \eqref{eq:extended-G-bound}, and
\eqref{eq:absolute-level-estimate}, and since $p,p'>1$, we finally get
\begin{multline*}
\left|\int_{\Omega_T}\varphi\,d|\mu|
\right|\le 
\int_{\Omega_T}|\varphi|\,d|\mu|
=
\int_0^\infty|\mu|(\{|\varphi|>\lambda\})\,d\lambda
\le
|\mu|(\Omega_T)
+
C\left(\int_{\Omega_*}\overline G\,dx\,dt\right)
\int_1^\infty\frac{1}{\lambda^{\min\{p,p'\}}}\,d\lambda
\le C.
\end{multline*}
Once again, we highlight the fact that $C=C(n,p,\Omega,T,\vartheta,M,|\mu|(\Omega_T))$ and it does not depend on $\varphi$. By homogeneity, we infer that \eqref{maingoal} holds for every $\varphi\in\mathscr{D}$. \\\\\noindent
\textbf{Step 2: passage to general functions.}
Let $u\in\WM(\Omega_T)$. By Lemma~\ref{lem:smooth-density-WM}, choose a sequence
$\varphi_j\in\mathscr D$ such that
\[
\varphi_j\to u
\qquad\text{strongly in }\WM(\Omega_T).
\]
Lemma~\ref{lem:WM-into-WD} gives strong convergence in $\WD$. By
Theorem~\ref{thm:DPP-representative}, after passing to a (not relabeled) subsequence,
\[
\varphi_j\to\widetilde u
\qquad
\capD\text{-quasi everywhere,}
\]
and by Proposition~\ref{prop:capacity-null-equivalence}, the convergence also holds
$\capM$-quasi everywhere. Proposition~\ref{prop:diffuse} therefore gives
\[
\varphi_j\to\widetilde u
\qquad
|\mu|\text{-almost everywhere.}
\]
Fatou's lemma and the above estimate for smooth functions imply
\[\int_{\Omega_T}|\widetilde u|\,d|\mu|
\le
\liminf_{j\to\infty}
\int_{\Omega_T}|\varphi_j|\,d|\mu|
\le
C\liminf_{j\to\infty}\homnorm{\varphi_j}
=
C\homnorm{u}.
\]
This proves \eqref{maingoal} via the triangle inequality for integrals for every $\varphi\in W_M(\Omega_T)$. \\\\
\noindent
\noindent\textbf{Step 3: well-definedness and linearity.}
From the previous step, $\widetilde u\in L^1(\Omega_T,|\mu|)$, and we may
formally define
\[
L_{\mu}(u)
:=
\int_{\Omega_T}\widetilde u\,d\mu,\quad
L_{|\mu|}(u)
:=
\int_{\Omega_T}\widetilde u\,d|\mu|.
\]
The definitions are independent of the quasi-continuous representative. Indeed, if $\widetilde u_1$ and $\widetilde u_2$ are two quasi-continuous
representatives of the same Lebesgue class, then
\[
\widetilde u_1=\widetilde u_2
\quad\capD\text{-quasi everywhere},
\]
hence also $\capM$-quasi everywhere and finally $|\mu|$-almost everywhere, in view of Proposition \ref{prop:diffuse}.
Thus, the functionals are well defined. Lemma~\ref{lem:linear-representatives} gives
\[
\widetilde{u+v}=\widetilde u+\widetilde v,
\qquad
\widetilde{\alpha u}=\alpha\widetilde u
\]
$|\mu|$-almost everywhere, so $L_{|\mu|}$ is linear. The same arguments apply to $L_\mu$.  
\end{proof}
\begin{proof}[Proof on general $(r,r^q)$-cylinders]
By Remark \ref{rem:centers}, it is enough to work
with $0<r\le1$. Assuming first that $q\le p$, we have $r^p\le r^q$, and therefore $Q_{r,r^p}(z)\subset Q_{r,r^q}(z)$. Hence
\[
|\mu|\bigl(Q_{r,r^p}(z)\cap\Omega_T\bigr)
\le
M r^{n+q-\vartheta}
=
M r^{n+p-\vartheta_p},
\qquad
\vartheta_p:=p-q+\vartheta.
\]
The condition $\vartheta<q$ is equivalent to $\vartheta_p<p$. If $q>p$, the time interval of $Q_{r,r^p}(z)$ has length
$2r^p$ and can be covered by at most $N(r)
\le Cr^{p-q}$ intervals of length $2r^q$. Therefore
$Q_{r,r^p}(z)$ can be covered by at most $Cr^{p-q}$ cylinders
$Q_{r,r^q}(z_j)$ with the same spatial radius. Using
\eqref{eq:intrinsic-morrey},
\[
\begin{aligned}
|\mu|\bigl(Q_{r,r^p}(z)\cap\Omega_T\bigr)
\le
Cr^{p-q}Mr^{n+q-\vartheta}
=
CMr^{n+p-\vartheta}.
\end{aligned}
\]
The conclusion follows by the case $q=p$. Observe that, by Remark \ref{rem:centers}, the Morrey condition could also be applied to the shifted centers $z_j\in\R^{n+1}\setminus\Omega_T$, at the cost of a fixed structural constant.
\end{proof}

\begin{Remark}[Sharpness of the Morrey threshold]
\label{rem:sharpness}
The strict condition $\vartheta<\min\{p,q\}$ cannot be relaxed in general when $q\geq p$. Indeed, let $p=n\geq2,\, x_0\in\Omega$ and let $J\Subset(0,T)$ be an open and nondegenerate (i.e. with positive Lebesgue measure) interval. Define the product measure
\[
\mu:=\delta_{x_0}\otimes\mathcal L^1\llcorner J.
\]
Then $\mu$ is a finite and positive Radon measure and, for every
$z=(x,t)\in\Omega_T$, every $r>0$ and $q>1$,
\[
\begin{aligned}
\mu\bigl(Q_{r,r^q}(z)\cap\Omega_T\bigr)
=
\delta_{x_0}\bigl(B_r(x)\cap\Omega\bigr)
\mathcal L^1\bigl(J\cap(t-r^q,t+r^q)\bigr)
\leq
2r^q.
\end{aligned}
\]
In particular, for the parabolic scale $q=p$, since $p=n$, we have that $\mu$ satisfies the parabolic Morrey condition at the endpoint $\vartheta=p$. More generally, if $q\geq p$, then $\mu$ satisfies the Morrey condition on $(r,r^q)$-cylinders with $\vartheta=p=\min\{p,q\}$. Nevertheless,
\[
\mu\notin W_M'(\Omega_T).
\]
Indeed, it is classical that a point has zero relative elliptic
$n$-capacity: $\operatorname{cap}^{\mathrm{ell}}_n(\{x_0\};\Omega)=0$; see, for instance, 
\cite[Chapter~4]{EG}. By \cite[Theorem~2.16]{DPP},
\[
\operatorname{cap}_D(\{x_0\}\times J)=0.
\]
Let $J_0\Subset J$ be a compact nondegenerate interval. By
monotonicity and Proposition~\ref{prop:capacity-null-equivalence},
\[
\operatorname{cap}_M(\{x_0\}\times J_0)=0.
\]
On the other hand,
\[
\mu(\{x_0\}\times J_0)
=
\mathcal L^1(J_0)
>
0.
\]
Lemma~\ref{lem:dual-measure-polar} therefore implies that
$\mu\notin W_M'(\Omega_T)$. Thus the endpoint $\vartheta=\min\{p,q\}$
may fail for every $q\geq p$. When $1<q<p$, the sufficient
condition furnished by the main theorem is $\vartheta<q$, but the
preceding example does not rule out the endpoint $\vartheta=q$.
We therefore make no sharpness claim in this regime.
\end{Remark}
\begin{Lemma}[positive dual measures do not charge polar compact sets]
\label{lem:dual-measure-polar}
Let $\mu$ be a finite and positive Radon measure which defines an element of
$\WM'(\Omega_T)$. Then
\[
\capM(K)=0
\quad\implies\quad
\mu(K)=0,
\]
for every compact set $K\Subset\Omega_T$.
\end{Lemma}
\begin{proof}
If $\capM(K)=0$, there exist a sequence $\varphi_j\in\mathscr D$ such that $\varphi_j\geq\chi_K$ in $\Omega_T$,
and $\EnergyM(\varphi_j)\to0$.
In particular, $\varphi_j\geq0$ in $\Omega_T$,
$\varphi_j\geq1$ on $K$ and by
Proposition \ref{prop:capacity-null-equivalence}, also
$\homnorm{\varphi_j}\to0$. Hence, by positivity
\[
0\leq\mu(K)\leq\int_{\Omega_T}\varphi_j\,d\mu\leq\|\mu\|_{W_M'}\|\varphi_j\|_{W_M}\to0.
\]
\end{proof}

\begin{proof}[Proof of Proposition \ref{prop:dimension-dual-measures}]

To prove the first assertion, by Definition \ref{HausdorffDim} of parabolic Hausdorff dimension, we have $\mathcal H_p^n(E)=0$ and consequently $\operatorname{cap}_M(E)=0$, by \cite[Theorem~1.2]{MS}.
If $\lambda$ were a finite and positive Radon measure in
$W_M'(\Omega_T)$, then
Lemma~\ref{lem:dual-measure-polar} would give $\lambda(E)=0$.
Since $\operatorname{spt}\lambda\subset E$, this implies $\lambda=0$. For the second assertion, if $\dim^{\mathcal P}_{\mathcal H,p}(E)>n$, we may choose
\[
n<s<\min\{\dim^{\mathcal P}_{\mathcal H,p}(E),n+p\},
\]
and then $\mathcal H_p^s(E)>0$. Since $E$, endowed with the restriction
of $d_p$, is a compact metric space, Frostman's lemma
\cite[Theorem~8.17]{PM} implies existence of a nonzero Radon measure
$\lambda_E$, supported on $E$, such that
\[
\lambda_E\bigl(Q_{r,r^p}(z)\bigr)
\leq
Cr^s
\qquad
\text{for }z\in \Omega_T,\,\,r>0.
\]
Setting $\vartheta:=n+p-s$, we have 
\[
\lambda_E\bigl(Q_{r,r^p}(z)\cap\Omega_T\bigr)\le \lambda_E\bigl(Q_{r,r^p}(z)\bigr)
\leq
Cr^{n+p-\vartheta},\qquad 0<\vartheta<p;
\]
Theorem~\ref{thm:intrinsic-duality} therefore gives $\lambda_E\in W_M'(\Omega_T)$. 
\end{proof}

\subsection{The supercritical case} \label{sec:supercritical}

Here we show that in the range $\frac{2n}{n+2}\le p<\infty$ one can avoid the density argument of
Lemma~\ref{lem:smooth-density-WM} and prove the duality estimate directly on
$\WD(\Omega_T)$. Indeed, Sobolev embedding gives $W_0^{1,p}(\Omega)\hookrightarrow L^2(\Omega)$, and therefore $\WD(\Omega_T)$ and
\[
\left\{
 u\in L^p(0,T;W^{1,p}_0(\Omega)):
 \partial_tu\in L^{p'}(0,T;W^{-1,p'}(\Omega))
\right\}
\]
are Banach spaces with equivalent norms. Moreover, \cite[Lemma~2.19]{MS} gives
\[
\|u\|_{L^\infty(0,T;L^2(\Omega))}^2
\le
2\|\partial_tu\|_{L^{p'}(0,T;W^{-1,p'}(\Omega))}
\|u\|_{L^p(0,T;W^{1,p}(\Omega))}
+
\frac{C}{T^{2/p}}
\|u\|_{L^p(0,T;W^{1,p}(\Omega))}^2.
\]
Consequently, $\WD(\Omega_T)=\WM(\Omega_T)$
as Banach spaces, with equivalent norms. The covering argument of Theorem~\ref{thm:intrinsic-duality} can therefore be
run with the $\WD$-norm, giving for every $\varphi\in\mathscr D$
\[
\big|L_{|\mu|}(\varphi)\big|=\left|\int_{\Omega_T}\varphi\,d|\mu|\right|
\le C\|\varphi\|_{\WD}.
\]
If $u\in\WD(\Omega_T)$, then \cite[Theorem 2.11]{DPP} replaces Lemma \ref{lem:smooth-density-WM}, providing a sequence
$\varphi_j\in\mathscr D$ such that
\[
\varphi_j\to u
\qquad\text{strongly in }\WD(\Omega_T).
\]
Arguing exactly as in the proof of Theorem \ref{thm:intrinsic-duality}, we infer that the same convergence, up to subsequences, holds $|\mu|$-almost everywhere, and we conclude again by means of Fatou's Lemma. Thus the density of $\mathscr D$ in $\WM$ is not needed in this range.

\section{Applications to parabolic equations}\label{sec:pde-consequences}
One of the main applications of Theorem \ref{thm:intrinsic-duality} concerns the theory of existence and uniqueness of solutions to the Cauchy-Dirichlet problem
\begin{equation}\label{eq:intro-parabolic-problem}
\begin{cases}
\partial_tu-\Div a(x,t,\nabla u)=\mu
    &\text{in }\Omega_T,\\
u=0 &\text{on }\partial\Omega\times(0,T),\\
u(0)=u_0 &\text{in }\Omega,
\end{cases}
\end{equation}
where $\Omega\subset\R^n$ is a bounded domain with Lipschitz boundary and $\Omega_T:=\Omega\times(0,T)$.  $\mu$ is a finite signed Radon measure on $\Omega_T$ and
$a:\Omega\times(0,T)\times\R^n\to\R^n$ is a Carathéodory vector field, i.e., the map $(x,t)\mapsto a(x,t,\xi) $ is measurable for every fixed $\xi\in\R^n$, and $\xi\to a(x,t,\xi)$ is continuous on $\R^n$ for almost every $(x,t)\in\Omega_T$. We assume that $a(x,t,\xi)$ satisfies for $1<p<\infty$ the following assumptions:
\begin{equation}\label{assumptionsa}
\begin{cases}
\langle a(x,t,\xi),\xi\rangle
\geq \alpha|\xi|^p-b(x,t)^{p'},\\[1mm]
|a(x,t,\xi)|
\leq \beta\bigl[b(x,t)+|\xi|^{p-1}\bigr],\\[1mm]
\langle a(x,t,\xi)-a(x,t,\eta),\xi-\eta\rangle
>0 \qquad \text{if }\xi\ne\eta,
\end{cases}
\end{equation}
for almost every $(x,t)\in\Omega_T$ and every $\xi,\eta\in\R^n$.
Here, $\alpha,\beta>0$ and $0\le b\in L^{p'}(\Omega_T)$; these assumptions, in particular, cover the $p$-Laplacian case, that is the model vector field appearing in \eqref{p.laplacian}, with $b(x,t)\equiv s^{p-1}$.
Under these assumptions, if $\mu\in L^{p'}(0,T;W^{-1,p'}(\Omega))$ and $u_0\in L^2(\Omega)$, then $\eqref{eq:intro-parabolic-problem}$ admits a unique energy solution $u\in W_M(\Omega_T)$, see \cite[Chapter II, Theorem 1.2 bis]{JLL}. 

\medskip

The purpose of this section is to explain what remains true for the larger class
of measure data obtained in the previous section.
Indeed, notice that $L^{p'}(0,T; W^{-1,p'}(\Omega))\hookrightarrow W_M'(\Omega_T)$ with strict inclusion; therefore, existence for a general right-hand side in
$W_M'(\Omega_T)$ does not follow directly from the classical Lions theorem. The appropriate variational framework relies on the following representation. By \cite[Lemma~2.24 and Remark~3.3]{DPP}, every $\mu\in W'_M(\Omega_T)$ admits a decomposition of the form 
\begin{equation}\label{eq:intro-DPP-decomposition}
\begin{aligned}
\langle \mu,\varphi\rangle_{W_M'(\Omega_T),W_M(\Omega_T)}
&=
\int_0^T\langle g_1(t),\varphi(t)\rangle\,dt
-
\int_0^T\langle \partial_t\varphi(t),g_2(t)\rangle\,dt
+
\int_{\Omega_T}g_3\varphi\,dx\,dt,
\end{aligned}
\end{equation}
where 
\[
g_1\in L^{p'}(0,T;W^{-1,p'}(\Omega)),
\qquad
g_2\in L^p(0,T;W^{1,p}_0(\Omega)),
\qquad
g_3\in L^{p'}(0,T;L^2(\Omega)).
\]
In this latter case, \cite[Theorem~3.1 and Remarks~3.2--3.3]{DPP} provide a unique
variational solution, but the solution itself need not belong to $\WM(\Omega_T)$. Under decomposition \eqref{eq:intro-DPP-decomposition}, the shifted function satisfies
\begin{equation*}
\partial_t (u-g_2)= \operatorname{div}a(x,t,\nabla u)+g_1+g_3 \in L^{p'}(0,T;W^{-1,p'}(\Omega))+L^{p'}(0,T;L^2(\Omega)),
\end{equation*}
and
\[
u-g_2\in C([0,T];L^2(\Omega)),\qquad
(u-g_2)(0)=u_0
\,\,\text{in }L^2(\Omega),
\]
see \cite[Remark 3.2]{DPP}. Starting from \cite[Theorem~3.1 and Remarks~3.2--3.3]{DPP},
we have the following result. 

\begin{Proposition}[Variational solutions for Morrey measures]\label{prop:DPP-variational-solution}
Assume that $\mu$ is a finite signed Radon measure satisfying any of the preceding hypotheses implying $\mu\in\WM'(\Omega_T)$, and let $u_0\in L^2(\Omega)$. Under hypothesis
\eqref{assumptionsa}, problem~\eqref{eq:intro-parabolic-problem} admits a unique
variational solution
\[
u\in L^p(0,T;W^{1,p}_0(\Omega)),
\]
with $u-g_2\in\WD(\Omega_T)$, where $g_2$ is taken from a fixed decomposition
\eqref{eq:intro-DPP-decomposition}, satisfying
\begin{equation*}
-\int_0^T
 \langle\partial_t\varphi,u\rangle\,dt
-\int_\Omega u_0(x)\varphi(x,0)\,dx
+
\int_{\Omega_T}\langle a(x,t,\nabla u),\nabla\varphi\rangle\,dx\,dt
=
\int_{\Omega_T}\widetilde\varphi\,d\mu,
\end{equation*}
for every $\varphi\in\WM(\Omega_T)$ such that $\varphi(T)=0$ in $L^2(\Omega)$.
\end{Proposition}

\begin{proof}
Theorem~3.1 of \cite{DPP} gives existence and uniqueness for right-hand sides
in $\WD'(\Omega_T)$. Remark~3.3 of the same paper states that the theorem can
be formulated with right-hand side in $\WtD'(\Omega_T)$ and test functions in
$\WtD(\Omega_T)$. Since $\WtD=\WM$, it applies to the present datum. The
measure acts through the quasi-continuous representative by
Theorem~\ref{thm:intrinsic-duality}.
\end{proof}
The variational solution provided by Proposition \ref{prop:DPP-variational-solution} need not belong to $W_M(\Omega_T)$: its relation to the energy framework is clarified by the preceding decomposition. In general, $u-g_2$ need not belong to $\WM$, because of the term
$g_3\in L^{p'}(0,T;L^2(\Omega))$. However, in the range $p\ge\frac{2n}{n+2}$ discussed in
Section~\ref{sec:supercritical}, the embedding
$L^2(\Omega)\hookrightarrow W^{-1,p'}(\Omega)$ and the preceding identity imply that $u-g_2\in\WM(\Omega_T)$. If the datum $\mu\in L^{p'}(0,T;W^{-1,p'}(\Omega))$, we can choose $g_2=g_3=0$ in \eqref{eq:intro-DPP-decomposition}, so that the shift
disappears and one recovers the classical energy solution of \cite{JLL}.

We finally turn to renormalized solutions. For soft measure data and every $u_0\in L^1(\Omega)$, existence and uniqueness follow from \cite[Theorem~1.3]{DPP}, extending the earlier $L^1$ theories
of Blanchard--Murat and Prignet and the measure-data approach of Boccardo--Gallou\"et; see \cite{BG,BM,P}. Following \cite[Definition~2.22]{DPP}, the class of soft measures
is defined by
\[
\mathcal M_0(\Omega_T):=\big\{\lambda\in\mathcal M_b(\Omega_T):\lambda(E)=0\text{ for every Borel set }E\subset\Omega_T\text{ with }\capD(E)=0\big\},
\]
where $\mathcal M_b(\Omega_T)$ denotes the space of finite
signed Radon measures on $\Omega_T$. This class
contains the classical $L^1$-data. Indeed, if
$f\in L^1(\Omega_T)$ and $\lambda:=f(x,t)\,dx\,dt$, then for every open set
$U\subset\Omega_T$ and every admissible function $v$ in the definition of
$\capD(U)$, we have
\[
|U|^{1/p}
\le\|v\|_{L^p(\Omega_T)}
\le C\|v\|_{\WD}.
\]
Taking the infimum and then the outer extension gives $|E|^{1/p}\le C\capD(E)$ for every Borel set $E$; therefore
\[
\capD(E)=0
\quad\implies\quad
\mathcal L^{n+1}(E)=0
\quad\implies\quad
\lambda(E)=\int_E f\,dx\,dt=0.
\]
Hence $L^1(\Omega_T)\subset\mathcal M_0(\Omega_T)$. The inclusion is strict,
because soft measures may be singular with respect to Lebesgue measure.

\begin{Proposition}[Renormalized solutions]
Assume that $\mu$ is a finite signed Radon measure satisfying any of the preceding hypotheses implying
$\mu\in\WM'(\Omega_T)$. Then, under
\eqref{assumptionsa}, for every $u_0\in L^1(\Omega)$,
problem~\eqref{eq:intro-parabolic-problem} admits a unique renormalized solution in the
sense of \cite[Theorem~1.3]{DPP}.
\end{Proposition}
\begin{proof}
Let $E\subset\Omega_T$ be Borel and assume that
$\capD(E)=0$. By Proposition~\ref{prop:capacity-null-equivalence}, we also get $\capM(E)=0$.
Proposition~\ref{prop:diffuse} therefore gives $|\mu|(E)=0$, and consequently $\mu(E)=0$.
Thus $\mu\in\mathcal M_0(\Omega_T)$ and the conclusion follows
from \cite[Theorem~1.3]{DPP}.
\end{proof}
\section{Appendix: proof of the density lemma}

We conclude the paper with the proof of
Lemma~\ref{lem:smooth-density-WM}. The argument combines the spatial
localization used in \cite[Lemma~2.22]{MS} with the  regularization procedure of
\cite[Lemma~A.1]{DPP}. We emphasize that
\cite[Theorem~2.11]{DPP} already gives density of $\mathscr{D}$ in the larger space
$W_D(\Omega_T)$. The additional point needed here is the strong
convergence of the time derivatives in $L^{p'}(0,T;W^{-1,p'}(\Omega))$.\\\\\noindent
\textbf{Step 1: reflection and spatial localization.}
By Lemma~\ref{lem:WM-into-WD} and the Lions--Magenes time-continuity theorem (see \cite[Remark~2.2]{DPP}),
$u$ has a representative in the space $C([0,T];L^2(\Omega))$. Applying Lemma~\ref{lem:time-reflection} with $R=T^{1/p}$, we extend
$u$ by even reflections to $I_*:=(-T,2T)$, and
we denote the extension by $\overline u$. It satisfies
\begin{equation*}
\|\overline u\|_{L^p(I_*;W^{1,p}(\Omega))}
+
\|\partial_t\overline u\|_{L^{p'}(I_*;W^{-1,p'}(\Omega))}
+
\|\overline u\|_{L^\infty(I_*;L^2(\Omega))}
\leq
C\|u\|_{W_M(\Omega_T)}.
\end{equation*}
Moreover $\overline u\in C(\overline{I_*};L^2(\Omega))$. Following the cutoff construction in
\cite[Lemma~2.22]{MS}, choose
$\xi_m\in C_c^\infty(\Omega)$ such that
\[
0\leq\xi_m\leq1,\quad
\xi_m=0
\,\,\text{if }\dist(x,\partial\Omega)\leq\frac1m,
\quad
\xi_m=1
\,\,\text{if }\dist(x,\partial\Omega)\geq\frac2m,\quad\text{and}\,\,
|\nabla\xi_m|\leq Cm.
\]
Set $u_m:=\xi_m\overline u$. We claim that
\begin{equation*}
u_m\longrightarrow\overline u
\qquad
\text{strongly in }
L^p(I_*;W_0^{1,p}(\Omega))\cap L^\infty(I_*;L^2(\Omega)),
\end{equation*}
and
\begin{equation}\label{eq:density-cutoff-derivative-short}
\partial_tu_m
\longrightarrow
\partial_t\overline u
\qquad
\text{strongly in }L^{p'}(I_*;W^{-1,p'}(\Omega)).
\end{equation}
For the spatial norm, observe that
\[
\nabla(u_m-\overline u)
=
(\xi_m-1)\nabla\overline u
+
\overline u\,\nabla\xi_m.
\]
The first term in the right-hand side converges to zero 
in $L^p(\Omega\times I_*)$ by dominated convergence. For the second
one, thanks to Hardy's inequality we have
\[
\int_{\Omega\times I_*}|\overline u\,\nabla\xi_m|^p\,dx\,dt
\leq
C\int_{\Omega\times I_*}\frac{|\overline u(\cdot,t)|^p}{|\dist(x,\partial\Omega)|^p}
\chi_{\{\dist(x,\partial\Omega)<2/m\}}\,dx\,dt\le  C\int_{\Omega\times I_*}|\nabla \overline{u}|^p\,dx\,dt.
\]
Therefore, we conclude again by dominated convergence.
To prove convergence in $L^\infty(I_*;L^2(\Omega))$, observe that
$1-\xi_m$ is supported in
\[
B_m:=\{x\in\Omega:\dist(x,\partial\Omega)<2/m\}. 
\]
For any $t \in I_*$, we can bound the $L^2$-norm as follows
\begin{equation}\label{boundH}
  \|u_m(t) - \bar{u}(t)\|_{L^2(\Omega)} = \|\bar{u}(t)(1 - \xi_m)\|_{L^2(\Omega)} \le \|\bar{u}(t) \chi_{B_m}\|_{L^2(\Omega)},  
\end{equation}
where $\bar u(t):=\bar u(\cdot,t)$. We need to prove that $\sup_{t \in I_*} \|\bar{u}(t) \chi_{B_m}\|_{L^2(\Omega)} \to 0$ as $m \to \infty$. Since $\bar{u}$ belongs to $C(\bar{I}_*; L^2(\Omega))$, and the time interval $\bar{I}_*$ is compact, its continuous image  $\mathcal{K} := \{\bar{u}(\cdot,t) : t \in \bar{I}_*\}$ is a compact subset of $L^2(\Omega)$. Fix $\varepsilon > 0$. By the compactness of $\mathcal{K}$, we can cover it with a finite number of open balls in $L^2(\Omega)$ of radius $\varepsilon$, centered at some elements $v_1, v_2, \dots, v_M \in \mathcal{K}$. 
Since there are only finitely many centers, we can find an integer $m_0$ large enough such that for all $m \ge m_0$, we simultaneously have
\[
    \|v_i \chi_{B_m}\|_{L^2(\Omega)} < \varepsilon \quad \text{for all } i = 1, \dots, M.
\]
Now, let $t \in I_*$ be an arbitrary time. Thus, there exists some index $i$ such that $\|\bar{u}(t) - v_i\|_{L^2(\Omega)} < \varepsilon$. Using the triangle inequality, we can finally estimate the norm on the strip $B_m$ as follows
\begin{align*}
    \|\bar{u}(t) \chi_{B_m}\|_{L^2(\Omega)} \le \|(\bar{u}(t) - v_i) \chi_{B_m}\|_{L^2(\Omega)} + \|v_i \chi_{B_m}\|_{L^2(\Omega)} 
    <  2\varepsilon.
\end{align*}
Since this bound holds for every $t \in I_*$ provided that $m \ge m_0$, taking the supremum over $t$ and recalling \eqref{boundH} yields
\[
    \|u_m - \bar{u}\|_{L^\infty(I_*; L^2(\Omega))} \le 2\varepsilon.
\]
By arbitrariness of $\varepsilon>0$, the uniform convergence follows. It remains to justify
\eqref{eq:density-cutoff-derivative-short}. For
$g\in W^{-1,p'}(\Omega)$, define $\xi_mg\in W^{-1,p'}(\Omega)$ by setting
\[
\langle\xi_mg,\phi\rangle_{W^{-1,p'}(\Omega),W_0^{1,p}(\Omega)}
:=
\langle g,\xi_m\phi\rangle_{W^{-1,p'}(\Omega),W_0^{1,p}(\Omega)},
\qquad
\phi\in W_0^{1,p}(\Omega).
\]
Hardy's inequality gives the uniform multiplier estimate
\begin{equation*}
\|\xi_m\phi\|_{W^{1,p}(\Omega)}
\leq
C\|\phi\|_{W^{1,p}(\Omega)},
\end{equation*}
and consequently
\begin{equation}\label{unif1}
\|\xi_mg\|_{W^{-1,p'}(\Omega)}
\leq
C\|g\|_{W^{-1,p'}(\Omega)}.
\end{equation}
Furthermore,
\begin{equation}\label{unif2}
\xi_mg\longrightarrow g
\qquad\text{strongly in }W^{-1,p'}(\Omega)
\end{equation}
for every $g\in W^{-1,p'}(\Omega)$. To see this, it is enough to use the density
of $C_c^\infty(\Omega)$ in $W^{-1,p'}(\Omega)$: for a compactly supported smooth
$g$, one has $\xi_mg=g$ for all sufficiently large $m$, and the
general conclusion follows from estimate \eqref{unif1}.
Recall that the density of $C_c^\infty(\Omega)$ in $W^{-1,p'}(\Omega)$ follows,
for instance, by writing $g=\operatorname{div}F$ with
$F\in L^{p'}(\Omega;\mathbb R^n)$ and approximating $F$ in
$L^{p'}$ by compactly supported smooth vector fields.
To conclude, we apply the spatial result \eqref{unif2} pointwise in time to $\partial_t \bar{u}(t) \in W^{-1,p'}(\Omega)$, obtaining
\begin{equation}\label{converg}
\|\xi_m \partial_t \bar{u}(t) - \partial_t \bar{u}(t)\|_{W^{-1,p'}(\Omega)} \to 0 \quad \text{pointwise a.e. in } I_*.
\end{equation}
Using the triangle inequality and \eqref{unif1}, we can construct an integrable bound for our sequence
\[
    \|\xi_m \partial_t \bar{u}(t) - \partial_t \bar{u}(t)\|_{W^{-1,p'}(\Omega)} \le \|\xi_m \partial_t \bar{u}(t)\|_{W^{-1,p'}(\Omega)} + \|\partial_t \bar{u}(t)\|_{W^{-1,p'}(\Omega)} \le C\|\partial_t \bar{u}(t)\|_{W^{-1,p'}(\Omega)}.
\]
Since $C^{p'}\|\partial_t \bar{u}(t)\|_{W^{-1,p'}(\Omega)}^{p'}$ is integrable over $I_*$, using  \eqref{converg}, we can apply the dominated convergence theorem to conclude that $\xi_m\partial_t\overline u=\partial_t u_m \to \partial_t \bar{u}$ in $L^{p'}(I_*; W^{-1,p'}(\Omega))$. This concludes the first step, and shows that
\begin{equation}\label{eq:density-cutoff-WM-short}
u_m\longrightarrow u
\qquad
\text{strongly in }W_M(\Omega_T)
\end{equation}
after restriction to $\Omega_T$.

\medskip
\noindent
\textbf{Step 2: regularization for fixed $m$.}
Fix $m$, choose $\chi\in C_c^\infty(I_*)$ such that
$\chi\equiv 1$ in a neighbourhood of $[0,T]$ and consider 
\[
w_m(x,t):=\chi(t)u_m(x,t).
\]
Since $\xi_m$ has compact support in $\Omega$, we have $\supp w_m\Subset\Omega\times I_*$, and $w_m=u_m$
in a neighbourhood of $\overline{\Omega_T}$. Let $\rho_\varepsilon$ be a standard space--time mollifier and,
after extending $w_m$ by zero, define $w_{m,\varepsilon}:=\rho_\varepsilon*w_m$. For $\varepsilon>0$ sufficiently small, $w_{m,\varepsilon}\in C_c^\infty(\Omega\times I_*)$, and hence
\[
\varphi_{m,\varepsilon}
:=
\left.w_{m,\varepsilon}\right|_{\Omega_T}
\in\mathscr D.
\]
Since $w_m$ has compact support in $\Omega\times I_*$, the
regularization result
\cite[Lemma~A.1]{DPP} gives
\[
w_{m,\varepsilon}
\longrightarrow w_m
\qquad
\text{strongly in }W_D(\Omega\times I_*).
\]
In particular,
\begin{equation}\label{eq:density-DPP-spatial-short}
\varphi_{m,\varepsilon}
\longrightarrow u_m
\qquad
\text{strongly in }L^p(0,T;W_0^{1,p}(\Omega)).
\end{equation}
Moreover, the continuous embedding $W_D(\Omega\times I_*)\hookrightarrow C(\overline{I_*};L^2(\Omega))$ yields
\begin{equation}\label{eq:density-DPP-H-short}
\varphi_{m,\varepsilon}
\longrightarrow u_m
\qquad
\text{strongly in }L^\infty(0,T;L^2(\Omega)).
\end{equation}
We finally check the strong convergence of the time derivatives.
Since $\partial_tw_m=a_m+b_m$, where
\[
a_m:=\chi\,\partial_tu_m
\in L^{p'}(I_*;W^{-1,p'}(\Omega)),\qquad
b_m:=\chi' u_m
\in L^{p'}(I_*;L^2(\Omega)),
\]
the proof of \cite[Lemma~A.1]{DPP}, applied separately to these two
components, gives
\[
\rho_\varepsilon*a_m
\longrightarrow a_m
\qquad
\text{strongly in }L^{p'}(I_*;W^{-1,p'}(\Omega)),
\]
as $\varepsilon\to 0^+$. On the other hand, $\supp b_m\subset\supp\chi'$, hence it
is separated by a positive distance from $[0,T]$. Hence, for
$\varepsilon$ sufficiently small, we have $(\rho_\varepsilon*b_m)|_{\Omega_T}=0$.
Since $\chi=1$ in a neighbourhood of $[0,T]$, one has
$a_m=\partial_tu_m$ on $\Omega_T$, and then
\begin{equation}\label{eq:density-DPP-time-short}
\partial_t\varphi_{m,\varepsilon}
\longrightarrow
\partial_tu_m
\qquad
\text{strongly in }L^{p'}(0,T;W^{-1,p'}(\Omega)).
\end{equation}

Combining
\eqref{eq:density-DPP-spatial-short},
\eqref{eq:density-DPP-H-short}, and
\eqref{eq:density-DPP-time-short}, we obtain, for every fixed $m$,
\begin{equation}\label{eq:density-fixed-m-short}
\varphi_{m,\varepsilon}
\longrightarrow u_m
\qquad
\text{strongly in }W_M(\Omega_T)\qquad\text{as $\varepsilon\to0^+$.}
\end{equation}

\medskip
\noindent
\textbf{Step 3: diagonal choice.}
By \eqref{eq:density-cutoff-WM-short}, choose a sequence $m_j\to\infty$ such that
\[
\|u_{m_j}-u\|_{W_M(\Omega_T)}
\leq
2^{-j}.
\]
For each $j$, use
\eqref{eq:density-fixed-m-short} to choose
$\varepsilon_j>0$ sufficiently small that
\[
\|\varphi_{m_j,\varepsilon_j}-u_{m_j}\|_{W_M(\Omega_T)}
\leq
2^{-j};
\]
setting $\varphi_j
:=
\varphi_{m_j,\varepsilon_j}\in\mathscr D$, we conclude that
\[
\|\varphi_j-u\|_{W_M(\Omega_T)}\leq2^{1-j}\longrightarrow0
\]
and this proves the density of $\mathscr D$ in $W_M(\Omega_T)$.

\subsection*{Acknowledgments}
The author is grateful to P. Baroni for several discussions on the subject, and helpful suggestions.
L.~Braglia is supported by the joint Ph.D.\ program of the
University of Ferrara, the University of Modena and Reggio Emilia,
and the University of Parma.

\subsection*{Declaration on the use of AI} The author used ChatGPT-5.6 Sol for language editing and checks of mathematical correctness and notational consistency; the mathematical arguments and conclusions were developed and verified by the author. The author takes full responsibility for the content of the manuscript including any remaining errors.


\begin{thebibliography}{HNVW}

\bibitem[A]{A}
D.~R.~Adams.
\newblock A note on Riesz potentials,
\newblock {\em Duke Math. J.} {\bf 42} (1975), 765--778.


\bibitem[AdSF]{AdSF}
M.~F. de Almeida and E.~P. dos Santos Filho.
\newblock Nonlinear parabolic thin sets and parabolic Wolff's
inequality.
\newblock \emph{Math. Ann.} \textbf{395} (2026), Paper No.~76.

\bibitem[AKP]{AKP}
B.~Avelin, T.~Kuusi, and M.~Parviainen.
\newblock Variational parabolic capacity.
\newblock \emph{Discrete Contin. Dyn. Syst.} \textbf{35} (2015),
no.~12, 5665--5688.

\bibitem[B14]{B14a}
P.~Baroni.
\newblock Marcinkiewicz estimates for degenerate parabolic equations
with measure data.
\newblock \emph{J. Funct. Anal.} \textbf{267} (2014), no.~9,
3397--3426.

\bibitem[B17]{B17}
P.~Baroni.
\newblock Singular parabolic equations, measures satisfying density
conditions, and gradient integrability.
\newblock \emph{Nonlinear Anal.} \textbf{153} (2017), 89--116.

\bibitem[BESS]{BESS}
M.~Borowski, T.~Elenius, L.~Sch\"atzler, and D.~Stolnicki.
\newblock Carleson-type removability for $p$-parabolic equations.
\newblock Preprint, \url{arXiv:2512.15277}, 2025.

\bibitem[BD]{BD}
T.~A. Bui and X.~T. Duong.
\newblock Global Marcinkiewicz estimates for nonlinear parabolic
equations with nonsmooth coefficients.
\newblock \emph{Ann. Sc. Norm. Super. Pisa Cl. Sci. (5)}
\textbf{18} (2018), no.~3, 881--916.

\bibitem[BG]{BG}
L.~Boccardo and T.~Gallou{\"e}t.
\newblock Non-linear elliptic and parabolic equations involving
measure data.
\newblock \emph{J. Funct. Anal.} \textbf{87} (1989), no.~1,
149--169.

\bibitem[BM]{BM}
D.~Blanchard and F.~Murat.
\newblock Renormalised solutions of nonlinear parabolic problems with
$L^1$ data: existence and uniqueness.
\newblock \emph{Proc. Roy. Soc. Edinburgh Sect. A}
\textbf{127} (1997), 1137--1152.

\bibitem[COV]{COV}
C.~Cascante, J.~M.~Ortega, and I.~E.~Verbitsky.
\newblock Trace inequalities of Sobolev type in the upper triangle
case.
\newblock \emph{Proc. Lond. Math. Soc. (3)}
\textbf{80} (2000), no.~2, 391--414.

\bibitem[DiB]{DiB}
E.~DiBenedetto.
\newblock \emph{Degenerate Parabolic Equations}.
\newblock Universitext.
\newblock Springer-Verlag, New York, 1993.

\bibitem[DPP]{DPP}
J.~Droniou, A.~Porretta, and A.~Prignet.
\newblock Parabolic capacity and soft measures for nonlinear
equations.
\newblock \emph{Potential Anal.} \textbf{19} (2003), no.~2,
99--161.

\bibitem[EG]{EG}
L.~C. Evans and R.~F. Gariepy.
\newblock \emph{Measure Theory and Fine Properties of Functions}.
\newblock Revised edition, Textbooks in Mathematics.
\newblock Chapman \& Hall/CRC, New York, 2015.

\bibitem[HW]{HW}
L.~I. Hedberg and T.~H. Wolff.
\newblock Thin sets in nonlinear potential theory.
\newblock \emph{Ann. Inst. Fourier (Grenoble)}
\textbf{33} (1983), no.~4, 161--187.

\bibitem[HNVW]{HNVW}
T.~Hyt\"onen, J.~van Neerven, M.~Veraar, and L.~Weis.
\newblock \emph{Analysis in Banach Spaces. Volume I:
Martingales and Littlewood--Paley Theory}.
\newblock Ergebnisse der Mathematik und ihrer Grenzgebiete.
3.~Folge. A Series of Modern Surveys in Mathematics, vol.~63.
\newblock Springer, Cham, 2016.

\bibitem[JLL]{JLL}
J.-L.~Lions.
\newblock \emph{Quelques m\'ethodes de r\'esolution des probl\`emes
aux limites non lin\'eaires}.
\newblock Dunod; Gauthier-Villars, Paris, 1969.

\bibitem[JL]{JL}
J.~L. Lewis.
\newblock Uniformly fat sets.
\newblock \emph{Trans. Amer. Math. Soc.}
\textbf{308} (1988), no.~1, 177--196.

\bibitem[KK]{KK}
J.~Kinnunen and R.~Korte.
\newblock Characterizations for the Hardy inequality.
\newblock In A.~Laptev, editor,
\emph{Around the Research of Vladimir Maz'ya I: Function Spaces},
volume~11 of \emph{International Mathematical Series},
pages 239--254.
\newblock Springer, New York, 2010.

\bibitem[KKKP]{KKKP}
J.~Kinnunen, R.~Korte, T.~Kuusi, and M.~Parviainen.
\newblock Nonlinear parabolic capacity and polar sets of
superparabolic functions.
\newblock \emph{Math. Ann.} \textbf{355} (2013), no.~4,
1349--1381.

\bibitem[KM13]{KM13}
T.~Kuusi and G.~Mingione.
\newblock Gradient regularity for nonlinear parabolic equations.
\newblock \emph{Ann. Sc. Norm. Super. Pisa Cl. Sci. (5)}
\textbf{12} (2013), no.~4, 755--822.

\bibitem[KM14]{KM14}
T.~Kuusi and G.~Mingione.
\newblock The Wolff gradient bound for degenerate parabolic equations.
\newblock \emph{J. Eur. Math. Soc.}
\textbf{16} (2014), no.~4, 835--892.

\bibitem[PM]{PM}
P.~Mattila.
\newblock \emph{Geometry of Sets and Measures in Euclidean Spaces}.
\newblock Cambridge Studies in Advanced Mathematics, vol.~44.
\newblock Cambridge University Press, Cambridge, 1995.

\bibitem[M07]{M1}
G.~Mingione.
\newblock The Calder\'on-Zygmund theory for elliptic problems with measure data.
\newblock \emph{Ann. Scuola Norm. Sup. Pisa Cl. Sci. (5)} \textbf{6} (2007), 195--261.

\bibitem[M11]{M2}
G.~Mingione.
\newblock Nonlinear measure data problems, 
\newblock \emph{Milan J. Math} \textbf{79} (2011), no. 2, 429--496.



\bibitem[MS]{MS}
K.~Moring and C.~Scheven.
\newblock On notions of $p$-parabolic capacity and applications.
\newblock \emph{Potential Anal.} \textbf{63} (2025), 895--946.

\bibitem[Pa]{Pa}
J.-T.~Park.
\newblock Marcinkiewicz regularity for singular parabolic
$p$-Laplace type equations with measure data.
\newblock \emph{Nonlinear Anal.} \textbf{223} (2022),
Paper No.~113073.


\bibitem[P]{P}
A.~Prignet.
\newblock Existence and uniqueness of ``entropy'' solutions of
parabolic problems with $L^1$ data.
\newblock \emph{Nonlinear Anal.} \textbf{28} (1997), 1943--1954.

\bibitem[V]{V}
I.~E. Verbitsky.
\newblock Nonlinear potentials and trace inequalities.
\newblock In J.~Rossmann et~al., editors, \emph{The Maz'ya Anniversary Collection}, vol.~110 of \emph{Operator Theory: Adv. Appl.}, pp. 323--343. Birkh\"auser, 1999.

\bibitem[Z]{Z}
W.~P. Ziemer.
\newblock \emph{Weakly Differentiable Functions:
Sobolev Spaces and Functions of Bounded Variation}.
\newblock Graduate Texts in Mathematics, vol.~120.
\newblock Springer-Verlag, New York, 1989.

\end{thebibliography}
\end{document}